\documentclass[11pt]{amsart}

\usepackage{enumerate,url,amssymb,  mathrsfs, graphicx, pdfsync}

\newtheorem{theorem}{Theorem}[section]
\newtheorem{lemma}[theorem]{Lemma}
\newtheorem*{lemma*}{Lemma}

\theoremstyle{definition}

\newtheorem{example}[theorem]{Example}

\newtheorem{question}[theorem]{Question}

\theoremstyle{remark}

\numberwithin{equation}{section}

\renewcommand{\bar}[1]{\overline{#1}}

\newcommand{\xx}{\mathbb{X}}
\newcommand{\yy}{\mathbb{Y}}
\newcommand{\cc}{\mathbb{C}}
\newcommand{\dd}{\mathbb{D}}
\newcommand{\rr}{\mathbb{R}}

\newcommand{\abs}[1]{\lvert#1\rvert}

\newcommand{\C}{\mathbb{C}}

\newcommand{\W}{\mathscr{W}}

\newcommand{\R}{\mathbb{R}}
\newcommand{\X}{\mathbb{X}}

\newcommand{\Y}{\mathbb{Y}}

\newcommand{\dtext}{\textnormal d}

\newcommand{\onto}{\xrightarrow[]{{}_{\!\!\textnormal{onto\,\,}\!\!}}}

\DeclareMathOperator{\diam}{diam}

\DeclareMathOperator{\re}{Re}
\DeclareMathOperator{\im}{Im}

\DeclareMathOperator{\loc}{loc}

\def\leq{\leqslant}
\def\geq{\geqslant}

\def\le{\leqslant}
\def\ge{\geqslant}

\def\XXint#1#2#3{{\setbox0=\hbox{$#1{#2#3}{\int}$}\vcenter{\hbox{$#2#3$}}\kern-.5\wd0}}

\def\XXiint#1#2#3{{\setbox0=\hbox{$#1{#2#3}{\iint}$}\vcenter{\hbox{$#2#3$}}\kern-.5\wd0}}

\begin{document}
\title{The Sobolev Jordan-Sch\"onflies Problem}

\author[A. Koski]{Aleksis Koski}
\address{Department of Mathematics and Statistics, P.O.Box 35 (MaD) FI-40014 University of Jyv\"askyl\"a, Finland}
\email{aleksis.t.koski@jyu.fi}

\author[J. Onninen]{Jani Onninen}
\address{Department of Mathematics, Syracuse University, Syracuse,
NY 13244, USA and  Department of Mathematics and Statistics, P.O.Box 35 (MaD) FI-40014 University of Jyv\"askyl\"a, Finland
}
\email{jkonnine@syr.edu}
\thanks{  A. Koski was supported by the Academy of Finland grant number 307023.
J. Onninen was supported by the NSF grant  DMS-1700274.}




\begin{abstract} We consider the planar unit disk $\mathbb D$ as the reference configuration and a Jordan domain $\Y$ as the deformed configuration, and
study the problem of extending a given boundary homeomorphism $\varphi \colon \partial \mathbb D \onto \partial \Y$ as a Sobolev homeomorphism of the complex plane. Investigating such a Sobolev variant of the classical  Jordan-Sch\"onflies  theorem is  motivated by the well-posedness of the related pure displacement
variational questions in the theory of Nonlinear Elasticity (NE)  and Geometric Function Theory (GFT). Clearly, the necessary condition  for the boundary mapping $\varphi$ to admit a $\W^{1,p}$-Sobolev homeomorphic extension is that it  first admits a continuous $\W^{1,p}$-Sobolev extension. For an arbitrary target domain $\Y$ this, however, is not sufficent. Indeed, {\bf first} for each $p<\infty$ we construct   a Jordan domain $\Y$ and a homeomorphism $\varphi \colon \partial \mathbb D \onto \partial \Y$ which admits a continuous $\W^{1,p}$-extension but does not even admit a
$\W^{1,1}$-homeomorphic extension.  {\bf Second}, for  a quasidisk target $\Y$ and  the whole range of $p$,  we prove that a boundary homeomorphism $\varphi \colon \partial \mathbb D \onto \partial \Y$ admits a $\W_{\loc}^{1,p}$-homeomorphic extension to $\C$ if and only if it admits a $\W^{1,p}$-extension to the unit disk.  Quasidisks have been a subject of intensive study in GFT.  They do not allow for singularities on the boundary such as cusps. {\bf Third}, for any power-type cusp target there is a boundary homeomorphism from  the unit circle whose harmonic extension has finite Dirichlet energy but does not  have a homeomorphic extension in $\W^{1,2} (\mathbb D, \C)$. Surprisingly, the Dirichlet integral ($p=2$) plays a unique role for the Sobolev  Jordan-Sch\"onflies   Problem in the case of cusp targets. Even more,  {\bf fourth} we prove that if  the target $\Y$ has piecewise smooth boundary,  $p\not =2 $ and $\varphi \colon \partial \mathbb D \onto \partial \Y$ has a  $\W^{1,p}$-Sobolev extension to $\mathbb D$, then it admits a homeomorphic extension to $\C$ in $\W_{\loc}^{1,p}(\C, \C)$.  {\bf Fifth}, if in addition $\Y$ is quasiconvex, then the one-sided  Sobolev  Jordan-Sch\"onflies problem has a solution when $p=2$. Indeed,  we show that the  harmonic extension of $\varphi \colon \partial \mathbb D \onto \partial \Y$ has a finite Dirichlet integral if and only if $\varphi$ admits a homeomorphic extension $h \colon \overline {\mathbb D} \onto \overline{\Y}$ with finite Dirichlet energy.
\end{abstract}

\maketitle
\section{Introduction}
Let $\mathbb D$ be the planar unit disk and $\varphi \colon \partial \mathbb D \to \mathbb C$ a topological embedding. The Jordan-Sch\"oenflies theorem states that there is a self-homeomorphism $h$ of the entire complex plane onto itself which  coincides with $\varphi$ on $\partial \mathbb D$. In particular, the set $\varphi (\partial \mathbb D)$ separates the plane into two domains, one bounded and the other unbounded. Throughout this text  $\Y\subset \C$ is a bounded Jordan domain  and $\varphi \colon \partial \mathbb D \onto \partial \Y$ a given boundary homeomorphism. The {\it Sobolev Jordan-Sch\"onflies  Problem} asks whether there exists a Sobolev  homeomorphism $h \colon \C \onto \C$ which coincides with $\varphi$ on $\partial \mathbb D$. Obviously,
for the boundary map $\varphi$ to admit a Sobolev homeomorphic extension it must first admit a Sobolev extension.
\vskip0.2cm
\begin{sjsp}  {\it For which $\Y \subset \C $ and $p\in [1, \infty]$ does every boundary homeomorphism $\varphi \colon \partial \mathbb D \onto \partial \mathbb Y$ that admits a continuous extension to $\overline{\mathbb D }$ in the Sobolev class $\W^{1,p}(\mathbb D, \C)$ also admit a homeomorphic extension $h \colon \C \onto \C$  in  $\W^{1,p}_{\loc} (\C, \C)$?} 
\end{sjsp}
\vskip0.2cm

It is worth noting  that the condition of a homeomorphism $\varphi \colon \partial \mathbb D \onto \partial \Y$ admitting a continuous  extension to $\overline{\mathbb D}$ in $\W^{1,p} (\mathbb D, \C)$ may be  characterized analytically. Indeed, for $1<p<\infty$,  the boundary mapping $\varphi $ admits a continuous $\W^{1,p}$-extension  if and only if it satisfies the so-called {\it $p$-Douglas condition},
\begin{equation}\label{eq:pDouglas}
\int_{\partial \mathbb D} \int_{\partial \mathbb D} \frac{\abs{\varphi(x) -\varphi(y)}^p  }{\abs{x-y}^p} \, \dtext x \, \dtext y < \infty \, , 
\end{equation}
for a proof we refer to~\cite[p. 151-152]{Stb}. The condition~\eqref{eq:pDouglas} is  known as the {\it Douglas condition}~\cite{Do} when  $p=2$.
Equivalently, $\varphi \colon \partial \mathbb D \onto \partial \Y$ satisfies the  $p$-Douglas condition if the $p$-harmonic extensions of both coordination functions $\im \varphi \colon \partial \mathbb D \to \R$ and    $\re \varphi \colon \partial \mathbb D \to \R$ belong to the Sobolev class $\W^{1,p} (\mathbb D, \R)$. The case $p=\infty$ on the other hand follows from the classical {\it Kirszbraun extension theorem}~\cite{Ki} which says that a mapping $\varphi \colon \partial \mathbb D \to \C$ has a Lipschitz extension to $\overline{\mathbb D}$ if and only if $\varphi$ is Lipschitz regular. In the other endpoint case $p=1$, according to {\it Gagliardo's theorem}~\cite{Ga}  a given $\varphi \colon \partial \mathbb D \onto \partial \Y$ has a Sobolev extension to $\mathbb D$ in $\W^{1,1} (\mathbb D , \mathbb C)$ exactly when $\varphi \in \mathscr L^1(\partial \mathbb D)$.  Analogously, we say that a boundary mapping $\varphi \colon \partial \mathbb D \to \C$ enjoys the {\it $1$-Douglas condition} provided $\varphi \in \mathscr L^1(\partial \mathbb D)$.

One of the reasons to study  the Sobolev Jordan-Sch\"onflies  problem  comes from the variational approach to Geometric Function Theory (GFT)~\cite{AIMb, IMb, Reb}, where the general framework of Nonlinear Elasticity (NE)~\cite{Anb, Bac, Cib}  is extremely fruitful and significant.  By the very assumptions of hyperelasticity, we  enquire into homeomorphisms $h  \colon  \mathbb D \onto \mathbb Y $ of smallest \textit{stored energy}
\begin{equation}\label{energ}
\mathsf E [h] =  \int_\mathbb D \mathbf {\bf E}(x,h, Dh )\,  \dtext x\, ,   \qquad  \mathbf E \colon  \mathbb D \times \mathbb Y \times \mathbb R^{2\times 2} 
\end{equation}
where   the so-called \textit{stored energy function} $\mathbf E$ characterizes the mechanical and elastic properties of the material occupying the domains.  We denote the class of   homeomorphisms $h \colon \mathbb D \onto \mathbb Y$  in the  Sobolev space $\mathscr W^{1,p} (\mathbb D, \mathbb C)$  by  $\mathscr H^{p}(\mathbb D, \mathbb Y)$ for $1\le p \le \infty$. In the related pure displacement variational questions one considers the class $\mathscr H_\varphi^p (\mathbb D , \mathbb Y)$ of Sobolev homeomorphisms $h \colon \overline{\mathbb D} \onto \overline{\Y}$ equal to  $\varphi$ on the boundary such that  $h\in \mathscr H^p (\mathbb D , \mathbb Y)$. 
One quickly runs into serious difficulties when passing to a weak limit of an
energy-minimizing sequence of $\W^{1,p}$-Sobolev homeomorphisms (injectivity can be lost)~\cite{IOan, IOinver}.  Therefore, in search for mathematical models of hyperelasticity, we must  adopt such limits as legitimate deformations and still comply, as much as possible, with the principle of {\it non-interpenetration of matter}. 
 When $p\ge 2$, an axiomatic assumption in the theories of  NE and GFT, such limits are monotone mappings~\cite{IOmoapprox, IOmo}.  Monotonicity, a concept by Morrey~\cite{Mor}, simply means that for a continuous $h \colon \overline{\mathbb D} \to \overline{\Y}$ the preimage $h^{-1} (y_\circ)$ of a point  $ y_\circ \in \overline{\Y}$ is a continuum in $\overline{\mathbb D}$.  Non-injective energy-minimal solutions, being monotone, may squeeze but not fold the $2D$-plates
or thin films.  In the case of   pure displacement  variational questions this naturally leads as a first step us to inquire whether the class  $\mathscr H_\varphi^p (\mathbb D , \mathbb Y)$ is nonempty. It is worth noting  that in the frictionless setting we always know that $\mathscr H^p (\mathbb D , \mathbb Y) \not= \emptyset$ when $p \le 2$ thanks to the Riemann Mapping Theorem.  
 
 Roughly speaking, prior to this paper the Sobolev Jordan-Sch\"onflies  Problem was understood only when $p<2$ or $p=\infty$. Here we focus on the remaining important cases $2\le p<\infty$.  Before giving a  detailed description  we  summarize the status of the problem  in the next table.

 \vskip0.3cm
  \begin{tabular}{ |p{3.0cm}||p{2.0cm}|p{2.0cm}|p{2.0cm}|p{1.3cm}|  }
 \hline
 \multicolumn{5}{|c|}{The known answers of  the  Sobolev Jordan-Sch\"onflies  Problem } \\
 \hline
$\partial \Y$& $1\le p <2$ & $p=2$ & $2<p<\infty$ & $p=\infty$\\
 \hline
Arbitrary     & Negative    &Negative&   Example~\ref{ex:zhangp}& Positive \\
Lipschitz graph   &  Positive    &  Positive &    Positive  &  Positive  \\
Quasicircle  &   Positive   & Theorem~\ref{thm:quasidisk}   & Theorem~\ref{thm:quasidisk}&  Positive \\
Piecewise smooth& Positive   & Example~\ref{ex:cuspp=2}&Theorem~\ref{thm:mainp>2}   &  Positive  \\
\hline
\end{tabular}
 \vskip0.3cm
 
The proofs of the previously known results rely mostly on careful study of analytical ways of extending the boundary map such as the harmonic extension and the Beurling-Ahlfors extension. It is, however, not to be expected that these methods are able to provide a complete picture of the problem. Hence we have turned to new, more direct methods of constructing homeomorphic extensions to prove our main theorems.

\subsection{Arbitrary target}

The Sobolev Jordan-Sch\"onflies  Problem is completely understood when $p=\infty$ due to a result of   Kovalev~\cite{Kovalev}.

\begin{theorem}\label{thm:lip}\textnormal{{\bf ($p=\infty$)}} Let $\varphi \colon \partial \mathbb D \to \cc$ be a Lipschitz embedding. Then $\varphi$  admits a homeomorphic Lipschitz extension to  $\C$. The Lipschitz constant of such an extension depends linearly on the Lipschitz constant of $\varphi$.
\end{theorem}

For an arbitrary $\Y$, the problem has no solution for $1 \le p \le  2$. Indeed,  Zhang~\cite{Zh}
constructed  a Jordan domain $\Y$ and a boundary homeomorphism
which admits a continuous $\W^{1,2}$-extension to $\overline{\mathbb D}$ but does not even admit a $\W^{1,1}$-homeomorphic extension to $\overline{\mathbb D}$. The boundary of the domain $\Y$ in
question is not rectifiable but does have Hausdorff dimension one. His construction relies on the Riemann Mapping Theorem and therefore   works only for $p\le 2$.  We  show that even replacing $2$ by any power $p<\infty$ does not guarantee  the existence of a homeomorphic extension in any Sobolev class.

\begin{theorem}\label{ex:zhangp} For $1\leq p<\infty$ there exists a Jordan domain $\Y$ and a homeomorphism $\varphi \colon \partial \mathbb D \onto \partial \Y$ which satisfies the $p$-Douglas condition (i.e. $\varphi$ admits a continuous $\W^{1,p}$-Sobolev extension)
 but does not admit a homeomorphic extension $h \colon \C \onto \C$ in the Sobolev class $\W^{1,1}_{\loc} (\C, \C)$.
 \end{theorem}

 \subsection{Target with Lipschitz boundary}
To begin with a classical result, notice that the theory of  of Rad\'o~\cite{Ra}, Kneser~\cite{Kn} and Choquet~\cite{Ch}  (RKC) solves the  Sobolev Jordan-Sch\"onflies  Problem when the target $\Y$ is convex and $p=2$. Indeed, first the RKC-theorem asserts among other things that if $\Y \subset \C$ is a bounded convex domain then the harmonic extension of a homeomorphism  $\varphi \colon \partial \mathbb D \onto \partial \Y$ is a diffeomorphism from $\mathbb D$ onto $\Y$. Second, since  the harmonic extension   minimizes the Dirichlet energy among all continuous Sobolev
mappings in $\W^{1,2}(\mathbb D, \C)$ equal to $\varphi$ on $\partial \mathbb D$ the harmonic extension belongs to $\W^{1,2} (\mathbb D, \mathbb C)$ if and only if  the map $\varphi $ admits a  finite Dirichlet energy extension to $\mathbb D$. If the target $\Y$, however, is not convex one can always construct  a boundary homeomorphism  $\varphi \colon \partial \mathbb D \onto \partial  \mathbb Y$ whose harmonic extension fails to be injective, see~\cite{AS, Kn}. Thus the Sobolev Jordan-Sch\"onflies problem already becomes nontrivial for a non-convex target and $p = 2$. For a domain $\yy$ with Lipschitz boundary it should however be noted that there exists a global bi-Lipschitz change of variables $\Phi \colon \C \onto \C$ for which $\Phi (\Y)$ is the unit disk. Naturally, the problem is invariant under such a global bi-Lipschitz change of variables. The above argument can be broadened to cover the entire range of $p<\infty$ and give a positive answer to the problem when $\Y$ has Lipschitz boundary. In fact, for $1\le p <2$, the classical RKC-theorem applies as the harmonic extension of any boundary homeomorphism $\varphi \colon \partial \mathbb D \onto \partial \mathbb D$ lies in the Sobolev space  $\W^{1,p} (\mathbb D , \C)$ for $p<2$, see~\cite{IMS, Ve}. For $p \in (2, \infty)$ we  may apply a $p$-harmonic variant of the RKC-theorem~\cite{AS}.   Here we also rely on  the fact that the variational formulation coincides with the classical formulation of the Dirichlet problem in any Jordan domain, see~\cite[\S 2.2]{KOext} for more details. 
Moving beyond targets with Lipschitz boundary, there however is no easy solution to the problem.

\subsection{Target with quasicircle  boundary} A {\it quasicircle} is the image of the unit circle under a quasiconformal self-homeomorphism
of $\C$.   The notion was  introduced independently by Pfluger~\cite{Pf} and Tienari~\cite{Ti}.  Recall that a $\W^{1,2}$-homeomorphism $f \colon \C \onto \C$ is quasiconformal if there is a constant $1 \le K < \infty$ such that
\[\abs{Df(x)}^2 \le K \det Df(x) \qquad \textnormal{a.e. in } \C \, . \]
Hereafter $\abs{\cdot}$ stands for the operator norm of matrices.
 In particular a quasicircle is a Jordan curve. The interior of a quasicircle is called a {\it quasidisk}. Quasidisks have been studied intensively for many years because of their exceptional functional theoretic properties, relationships with Teichm\"uller theory and Kleinian groups and interesting applications in complex dynamics, see~\cite{GeMo} for a  survey.  
Complex dynamics ({\it Julia sets} of rational maps, limit sets of quasi-Fuchsian groups) provide a  rich  source  of  examples  of  quasicircles  with  Hausdorff dimension greater than one. The Hausdorff dimension  of quasicircles may actually take any value in the interval $[1,2)$, see~\cite{GeVa}.   Perhaps the best know geometric characterization for a quasicircle is the  {\it Ahlfors' condition}~\cite{Ahre}.
It says that a planar Jordan curve $\mathcal C$  is a quasicircle if and only if   there is a constant $1 \leqslant \gamma\ <\infty$ such that for each pair of distinct points $a, b \in \mathcal C$ we have
\begin{equation}\label{eq:ah} \diam \Gamma   \le \gamma 
\abs{a-b}\end{equation}
where $\Gamma$ is the component of $\mathcal C \setminus \{a, b\}$ with smallest diameter. Equivalently~\eqref{eq:ah} can be given in terms of a reverse triangle inequality for three points: there is a constant $C$ such that  if a point $c\in \Gamma$, then
\begin{equation}\label{eq:threept}
\abs{a-c} + \abs{c-b}\le C\, \abs{a-b} \, . 
\end{equation}
This property is also called {\it bounded turning} condition, see~\cite{LVb}.
We proved with  Koskela~\cite{KKO} that if $1\le p < 2$ and $\partial \Y$ is a quasicircle, then any homeomorphism $\varphi \colon \partial \mathbb D \onto \partial \Y$ admits a homeomorphic extension $h \colon \C \onto \C$ in the Sobolev class $\W^{1,p}_{\loc} (\C , \C)$.  In particular, an arbitrary boundary homeomorphism $\varphi \colon \partial \mathbb D \onto \partial \Y$ satisfies the $p$-Douglas condition for $p<2$. The main point in our argument is a weighted homeomorphic extension theorem of the unit disk onto itself, see~\cite{KKO} for details.  
 This approach heavily relies on the fact that $p<2$ and cannot be extended to cover even the case $p=2$. Here we give a new direct way to construct Sobolev homeomorphic extensions of $\varphi \colon\partial  \mathbb D \onto \partial \Y$
and cover the entire  range of Sobolev exponents.
\begin{theorem}\label{thm:quasidisk}
Let $1\le p < \infty$ and $\partial \Y$ be a quasicircle. If a homeomorphism $\varphi \colon \partial \mathbb D \onto \partial \Y$ satisfies the $p$-Douglas condition (i.e. admits a Sobolev extension to $\mathbb D$ in $\W^{1,p} (\mathbb D, \C)$), then it admits a homeomorphic extension $h \colon \C \onto \C$ in the Sobolev class $\W_{\loc}^{1,p} (\C, \C)$.
\end{theorem}
Note that if $p<2$, then any boundary homeomorhism onto a quasicircle satisfies the $p$-Douglas condition.

\subsection{Target with piecewise smooth boundary}
The most standard examples of singular
boundaries which fail to satisfy the Ahlfors' condition~\eqref{eq:ah} are cusps.
Let $\Omega_\beta$ be an  inward cusp domain where the cusp is formed by the graph of the function $x \mapsto \abs{x}^\beta$ near $0$,  $\beta >1$, and a smooth curve, see Figure~\ref{fig:innercusp}.
\begin{figure}[htbp]
\centering
\includegraphics[width=0.25\textwidth]
{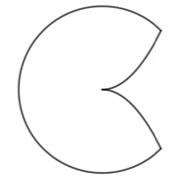}
\caption{An  inward  cusp domain, $\Omega_\beta$.}
\label{fig:innercusp}
\end{figure}
For a homeomorphism $\varphi \colon \partial \mathbb D \onto \partial \Omega_\beta$ there exists a homeomorphic extension $h 
\colon \C \onto \C$ which belongs to the Sobolev class $\W_{\loc}^{1,p} (\C , \C)$ for every $p<2$. This follows from our solution~\cite{KOext} to the Sobolev Jordan-Sch\"onflies  Problem when $p<2$ and the boundary of target domain being rectifiable. 
If $p=2$, however, the problem fails to have a solution for every power-type cusp target, independently of the sharpness of cusp.
\begin{example}\label{ex:cuspp=2}
 Let $\Omega_\beta$ be an  inward cusp domain for some $\beta>1$. Then  there is a boundary homeomorphism $\varphi \colon \partial \mathbb D \onto \partial \Omega_\beta$ which satisfies the Douglas condition (and hence admits a continuous  $\W^{1,2}$--extension) but does not admit a homeomorphic extension $h \colon \overline{\mathbb D} \onto \overline{\Omega_\beta}$ in the Sobolev class  $\W^{1,2} (\mathbb D, \C)$.
\end{example}
Surprisingly, $p=2$ is the only case when the problem does not have a solution. Indeed, we have
\begin{theorem}
 Let $\Omega_\beta$ be a cusp domain with $\beta>1$. If a homeomorphism  $\varphi \colon \partial \mathbb D \onto \partial \Omega_\beta$ satisfies the $p$-Douglas condition  for some $p>2$, then a homeomorphic extension $h\colon \C \onto \C$ of $\varphi$ lies in $\W_{\loc}^{1,p} (\C , \C)$.
\end{theorem}
We obtain this  as a corollary of our next result.
\begin{theorem}\label{thm:mainp>2}
Let $\Y$ be a domain with piecewise smooth boundary. Then a homeomorphism  $\varphi \colon \partial \mathbb D \onto \partial \Y$ which   satisfies the $p$-Douglas condition  for some $p \not =2$ admits  a homeomorphic extension $h\colon \C \onto \C$  in $\W_{\loc}^{1,p} (\C , \C)$.
\end{theorem}
Here and in what follows we say that a planar domain $\Omega$ has {\it piecewise smooth boundary} if $\partial \Omega = \bigcup^m_{j=1} \Gamma_j$, where each $\Gamma_j$ is a $\mathscr C^2$-regular curve. As mentioned earlier, the result of Theorem~\ref{thm:mainp>2} was already obtained in~\cite{KOext} for targets with rectifiable boundaries when $p < 2$. We do not know if Theorem~\ref{thm:mainp>2} is also true for a target $\Y$ with rectifiable boundary when $p>2$. 

Let us return to Example~\ref{ex:cuspp=2}. The key to our construction  is a careful analysis of  the modulus of continuity of the mappings in question.   Recall that,
a modulus of continuity of a mapping $f \colon \X \to \C$ is a function $\omega_f \colon [0, \infty) \to [0, \infty)$ if $\abs{f(x_1) -f(x_2)} \le \omega_f \big(\abs{x_1-x_2} \big)$ for all $x_1, x_2 \in \X$. 
For  a homeomorhism $h \in \W^{1,2}_{\loc} (\C , \C)$ we have
\begin{equation}\label{eq:modcont}
\int_0^1 \frac{[\omega_h(t)]^2}{t} \, \dtext t  <\infty \, . 
\end{equation}
In Example~\ref{ex:cuspp=2}, we construct a boundary homeomorphism $\varphi \colon \partial \mathbb D \onto \partial \Omega_\beta$ which satisfies the Douglas condition and fails to have the modulus of continuity estimate given by~\eqref{eq:modcont}. Therefore, clearly there is no homeomorphism $h \colon \C \onto \C$ in $\W^{1,2}_{\loc} (\C , \C)$ which coincides with $\varphi$ on $\partial \mathbb D$. The next result shows that the modulus of continuity  provided with~\eqref{eq:modcont} is not only necessary for a boundary homeomorphism $\varphi \colon \partial \mathbb D \onto \partial \Omega_\beta$  to have a homeomorphic extension in $\W^{1,2}_{\loc}(\C , \C)$ but also sufficient.
\begin{theorem}\label{thm:p=2}
Let $\Y$ be a domain with piecewise smooth boundary. Then  a homeomorphism  $\varphi \colon \partial \mathbb D \onto \partial \Y$ which   satisfies the Douglas condition admits  a homeomorphic extension $h\colon \C \onto \C$  in $\W_{\loc}^{1,2} (\C , \C)$ if and only if the boundary  homeomorphism  $\varphi$ has a modulus of continuity $\omega_\varphi$ which satisfies 
\begin{equation}\label{eq:modcontphi}
\int_0^1 \frac{[\omega_\varphi (t)]^2}{t} \, \dtext t  <\infty \, . 
\end{equation}
\end{theorem}

\subsection{The one-sided Sobolev Jordan-Sch\"onflies Problem}
As shown in Example~\ref{ex:cuspp=2}, there is a boundary homeomorphism $\varphi \colon \partial \mathbb D \onto \partial \Omega_\beta$ satisfying the Douglas condition which  does not admit a homeomorphic extension from $\overline{\mathbb D}$ onto $\overline{\Omega_\beta}$ with finite Dirichlet energy. However, it is still important to investigate under which conditions on the target such a one-sided Sobolev extension exists.

A careful examination of Example~\ref{ex:cuspp=2} reveals  that a potential reason why there is no such an extension lies in the fact that the internal distance in $\Omega_\beta$ of a pair of points on the cusp is not comparable to their Euclidean distance. It is hence expected that quasiconvexity of the target may be necessary to overcome these difficulties.  A domain $K \subset \C$ is {\it quasiconvex}  if each pair of points can be joined by a quasiconvex path. That is, there exists a constant $c\ge 1$ such that for all points $x,y\in K$ there exists a rectifiable path $\gamma$ joining $x$ and $y$, and satisfying
\begin{equation}\label{eq:qc}
\abs{\gamma} \le c \abs{x-y} \, . 
\end{equation}
Here $\abs{\gamma}$ stands for the length of the quasiconvex path $\gamma$. The notion of quasiconvexity plays a prominent role in GFT, see e.g.~\cite{GeMo, Ge2, NV, Va1.5} and the references mentioned therein.

An example of such a domain is the complementary domain of $\Omega_\beta$.  For a precise formulation, let $\mathcal K_\beta$ be an  outer cusp domain where the cusp is formed by the graph of the function $x \mapsto \abs{x}^\beta$ near $0$,  $\beta >1$, and a smooth curve, see Figure~\ref{fig:outercusp}.
\begin{figure}[htbp]
\centering
\includegraphics[width=0.25\textwidth]
{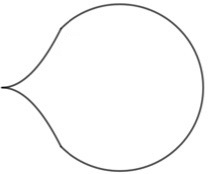}
\caption{An outer  cusp domain, $\mathcal K_\beta$.}
\label{fig:outercusp}
\end{figure}

Now,  for the cusp domains the question reads as follows.
\begin{question} Suppose that a homeomorphism $\varphi \colon \partial \mathbb D \onto \partial \mathcal K_\beta$ satisfies  the Douglas condition. Does $\varphi$  admit a homeomorphic extension $h \colon \overline{\mathbb D} \onto \overline{\mathcal K_\beta}$  in  $\W^{1,2}(\mathbb D, \C)$? 
\end{question}
An answer to this question follows as a corollary from our last theorem.
\begin{theorem}\label{thm:mainonesided}
Let $\Y$ be a quasiconvex domain with piecewise smooth boundary. Then a boundary homeomorphism $\varphi \colon \partial \mathbb D \onto \partial \Y$ admits a Sobolev homeomorphic extension $h \colon \overline{\mathbb D} \onto \overline{\Y}$ with finite Dirichlet integral $\int_\mathbb D \abs{Dh}^2 <\infty$ if and only if $\varphi$ satisfies the Douglas condition.
\end{theorem}

\section{Extending to a target with piecewise smooth boundary}

 \noindent In this section we prove Theorem \ref{thm:mainp>2} and Theorem \ref{thm:p=2}.

\begin{proof}[Proof of Theorem \ref{thm:mainp>2}] We may and do assume that $p>2$ because Theorem~\ref{thm:mainp>2} was already proved in~\cite{KOext}  when $p<2$. We suppose that $\varphi \colon  \partial \dd \to \partial \Y$ satisfies the $p$-Douglas condition and construct the required extension of $\varphi$ to $\dd$ via several steps. The proof will then be completed by using a reflection argument to extend $\varphi$ to $\cc \setminus \dd$ as well.\\\\
\emph{Step 1.} Reducing the problem.\\\\
By definition, $\partial \Y$ splits into a finite collection of curves that are locally the graphs of smooth functions and their intersection points which are the endpoints of two such curves. it is only necessary to construct the required extension in some small neighborhoods  of these intersection points. Indeed, for each such intersection point which is a common point of two $\mathscr C^2$-smooth pieces $\gamma_1$ and $\gamma_2$ of $\partial \Y$, we may separate this intersection point from the rest of the boundary via a smooth crosscut in $\Y$ that starts from a smooth point on $\gamma_1$, ends on a smooth point on $\gamma_2$, and forms any angles of our choice with respect to $\gamma_1$ and $\gamma_2$. By choosing these angles to be positive we can separate the domain $\Y$ into a finite collection of such neighborhoods and one domain which is a collection of smooth curves intersecting at positive angles, so it is bilipschitz-equivalent to the unit disk. See Figure \ref{smoothpieces} for an illustration.

\begin{figure}[h]
\includegraphics[scale=0.4]{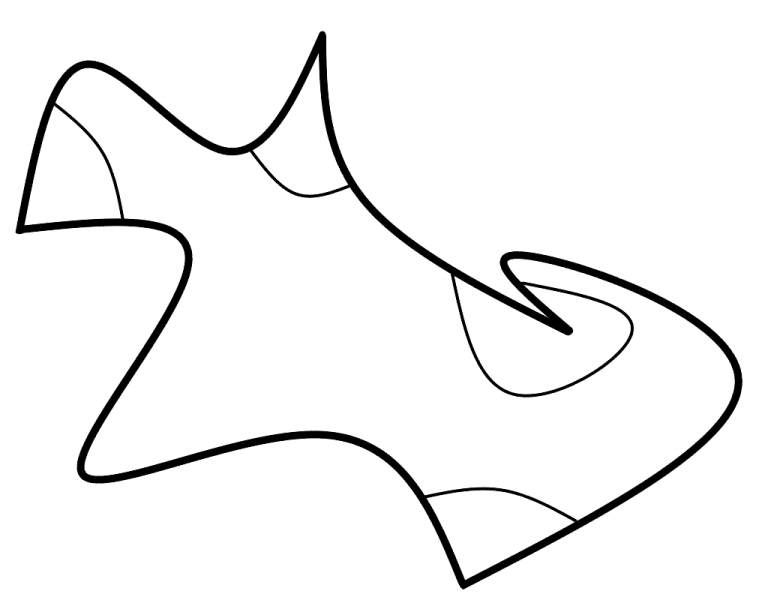}
\caption{The domain $\Y$, split into smaller pieces by smooth crosscuts.}\label{smoothpieces}
\end{figure}

Hence it is enough to consider  $\Y$ as the following domain. The boundary $\partial \Y$ consists of three smooth curves $\Gamma_1$, $\Gamma_2$ and $\Gamma_3$, where $\Gamma_3$ intersects the other two curves at a positive angle of our choice. The curves $\Gamma_1$ and $\Gamma_2$ can be assumed to intersect at an angle of zero, as the other case is trivial. Let  us denote their intersection point by $P$. Both of these curves can be assumed to be the graphs of $\mathscr C^2$-smooth functions in some coordinate system. As for the boundary map $\varphi: \partial \dd \to \partial \Y$, we are to assume that $\partial \dd$ splits into three arcs $I_1,I_2$ and $I_3$, each the preimage of one of the respective curves $\Gamma_i$ under $\varphi$. The map $\varphi$ is given from $I_1 \to \Gamma_1$ and $I_2 \to \Gamma_2$ and satisfies the $p$-Douglas condition on these parts, but from $I_3 \to \Gamma_3$ we must explain how to define the map in order that it will admit a $\mathscr W^{1,p}$-homeomorphic extension to $\Y$. This is done in a later step.\\\\
\emph{Step 2.} Boundary diffusion.\\\\
We concentrate on the map $\varphi_1 = \varphi\vert_{I_1} : I_1 \to \Gamma_1$. We abuse the notation a bit and consider $I_1$ as the unit interval $[0,1]$ and choose the parametrization so that $\varphi(0) = P$ is the intersection point of $\Gamma_1$ and $\Gamma_2$.

Let  $S = \{(x,y) \in \cc : 0 \leq x \leq 1,\, 0 \leq y \leq x\}$ be the unit triangle in the plane. We define a map $H : S \to \Gamma_1$, called the \emph{boundary diffusion}, with the following properties.
\begin{enumerate}
\item\label{it:1} $H$ lies in $\mathscr W^{1,p}(S)$.
\item\label{it:2} For each $t \in [0,1]$, if $L_t$ denotes the line segment between $(0,0)$ and $(1,t)$, the map $H$ takes $L_t$ homeomorphically onto $\Gamma_1$.
\item\label{it:3} $H\vert_{L_0} = \varphi_1$.
\item\label{it:4} $H$ is a smooth map on $L_1$.
\end{enumerate}
First, since $\Gamma_1$ is a smooth graph we may deform it via a global diffeomorphism and for the rest of the construction of $H$ we may assume that $\Gamma_1$ is the unit interval $[0,1]$  with the intersection point $P$ at $0$. The boundary values of $H$ are now easy to define. On $L_0 = [0,1]$, we define $H$ equal to $\varphi_1$. On the segment from $(0,1)$ to $(1,1)$ we define $H$ as the constant function $1$, and on $L_1$ we simply take $H$ as the projection to the $x$-axis. We then define $H$ inside of $S$ as the $p$-harmonic extension of these boundary values. To show that $H$ lies in $\mathscr W^{1,p} (S)$, it then simply remains to verify that the boundary values $H \vert_{\partial S}$ satisfy the $p$-Douglas condition
\begin{equation}\label{pdouglastriangle}
\int_{\partial S} \int_{\partial S} \frac{\left|H(x) - H(y)\right|^p}{\left|x-y\right|^p} \dtext y \, \dtext x \, < \, \infty.
\end{equation}
We split this integral into multiple parts depending on which side of $\partial S$ the points $x$ and $y$ lie. If $x$ and $y$ lie on the same side, the integral is easy to control since on two of the sides $H$ is either a constant or linear function and on the last one it is equal to $\varphi_1$ which was already assumed to satisfy the $p$-Douglas condition. The only nontrivial cases are when $x \in L_0$ and $y$ lies on one of the other sides. Both of these cases are dealt with in the same way so let us assume that $y$ lies on $L_1$. Then we estimate as follows:
\begin{align*}
\int_{[0,1]} \int_{L_1} \frac{\left|H(x) - H(y)\right|^p}{\left|x-y\right|^p} dy \, dx \, &\leq \int_{[0,1]} \int_{L_1} \frac{\left|H(x)-H(0)\right|^p}{\left|x-y\right|^p} dx \, dy \,
\\& =  \int_{[0,1]} \left|H(x)\right|^p \left( \int_{L_1} \frac{1}{\left|x-y\right|^p} dy \right) dx \,
\\& \leq C \int_{[0,1]} \frac{\left|H(x)\right|^p}{\left|x\right|^{p-1}} dx \,.
\end{align*}
It remains to show that the last integral is finite, using the fact that $H$ satisfies the $p$-Douglas condition from $[0,1]$ onto itself (since  $H=\varphi_1$ on $[0,1]$). Our proof for this claim is quite complicated. We state the claim as the following lemma. Note also that this lemma is the only part where the assumption $p > 2$ is used.
\begin{lemma}\label{lem21} Suppose that $f : [0,1] \to [0,1]$ is homeomorphic, $f(0) = 0$ and $f$ satisfies the $p$-Douglas condition on $[0,1]$ with $p > 2$. Then
\[\int_0^1 \frac{\left|f(x)\right|^p}{\left|x\right|^{p-1}} dx \, < \infty.\]
\end{lemma}

\begin{proof}[Proof of Lemma~\ref{lem21}]
Let us denote by $U_j = [2^{-(j+1)},2^{-j}]$ the dyadic intervals for $j = 0,1,2,\ldots$. We know from the $p$-Douglas condition that
\[\sum_{j=1}^\infty \int_{U_{j+1}} \int_{U_{j-1}} \frac{\left|f(x) - f(y)\right|^p}{\left|x-y\right|^p} dy \, dx \, < \, \infty.\]
\end{proof}
For $x \in U_{j-1}$ and $y \in U_{j+1}$ we estimate that $ 2^{-j-1} \le |x-y| \le  2^{-j+1}$ and $|f(x) - f(y)| \geq |f\left(2^{-j}\right) - f\left(2^{-(j+1)}\right)|$. Setting $a_j := f(2^{-j})$ we thus obtain
\[\sum_{j=0}^\infty 2^{(p-2)j} (a_j - a_{j+1})^p \, < \, \infty.\]
Define $\alpha = 2^{(p-2)/(2p)}$, from which we see that $\alpha > 1$ and $2^{(p-2)j} = \alpha^{2pj}$. Furthermore, define $c_j$ such that
\[c_j^p = (a_j - a_{j+1})^p \alpha^{2pj}.\]
Hence $\sum_{j=1}^\infty c_j^p < \infty$. Moreover, we see that $a_j - a_{j+1} = c_j\alpha^{-2j}$. Since $a_j \to 0$ as $j \to \infty$, we may sum this up to get that
\[a_k = \sum_{j=k}^\infty a_j - a_{j+1} = \sum_{j=k}^\infty c_j \alpha^{-2j}.\]
We now estimate via H\"older's inequality that
\begin{align*}
\sum_{k=0}^\infty a_k^p \alpha^{2pk} &= \sum_{k=0}^\infty \left(\sum_{j=k}^\infty c_j \alpha^{-2j}\right)^p \alpha^{2pk}
\\&\leq \sum_{k=0}^\infty \left(\sum_{j=k}^\infty c_j^p \alpha^{-pj}\right)\left(\sum_{j=k}^\infty \alpha^{-\frac{p}{p-1}j}\right)^{p-1} \alpha^{2pk}
\\&= C \sum_{k=0}^\infty \left(\sum_{j=k}^\infty c_j^p \alpha^{-pj}\right)\left( \alpha^{-\frac{p}{p-1}k}\right)^{p-1} \alpha^{2pk}
\\&= C \sum_{k=0}^\infty \sum_{j=k}^\infty c_j^p \alpha^{p(k-j)}
\\&= C_2 \sum_{j=0}^\infty c_j^p \sum_{k=0}^j \alpha^{p(k-j)}
\\&\leq C_2 \sum_{j=0}^\infty c_j^p \\&< \infty.
\end{align*}
Hence we have shown that $\sum_{k=0}^\infty a_k^p 2^{(p-2)k} < \infty$. Now if $x \in U_k$, then $f(x) \leq a_k$. Hence
\[\int_0^1 \frac{\left|f(x)\right|^p}{\left|x\right|^{p-1}} dx = \sum_{k=0}^{\infty}\int_{U_k} \frac{\left|f(x)\right|^p}{\left|x\right|^{p-1}} dx \leq C\sum_{k=0}^\infty a_k^p 2^{-k} \frac{1}{2^{-k(p-1)}} < \infty.\]
This finishes the proof of Lemma~\ref{lem21}. \qed\\

Returning to the proof of Theorem \ref{thm:mainp>2}, we have now shown all but one of the claimed properties of the boundary diffusion $H$. It remains to address the second point of showing that $H$ takes $L_t$ homeomorphically onto $\Gamma_1 = [0,1]$.

First of all, since $H$ is chosen as the coordinate-wise $p$-harmonic extension of its boundary values on $\partial S$. The  $p$-harmonic energy of $H=u+iv \colon S \to \mathbb C$ is defined by
\[\int_S \left(\abs{\nabla u}^p + \abs{\nabla v }^p \right) \, .\]
The $p$-harmonic mapping $H$  has the smallest  $p$-energy of all extensions. Since $H$ maps the endpoints of $L_t$ to $0$ and $1$ due to the chosen boundary values, the image of $L_t$ under $H$ must be the whole unit interval $[0,1]$ by continuity. From this we can infer that $H$ must also map $L_t$ increasingly onto $[0,1]$ since otherwise we could redefine $H$ on each $L_t$ as the smallest increasing replacement of $H\vert_{L_t}$ and this would yield a map of strictly smaller $p$-energy on $S$.

While this does not yet show that $H$ is injective on each $L_t$, we will argue that after a minor modification of $H$ we obtain a map with all of the desired properties. First, due to the classical regularity results for $p$-harmonic functions we find that $H$ is $\mathscr C^1$ in the interior of the triangle $S$, see~\cite{Ur}. Letting $\partial_{r}$ denote the directional derivative in the radial direction, which is also the direction of the segments $L_t$, we find that the set $S_+ = \{z \in S : \partial_r H(z) > 0\}$ must be open. For each $t$, there must be at least one point $z_t$ on $L_t$ which belongs to $S_+$ since $H$ maps $L_t$ to the whole unit interval. Let $r_t > 0$ denote a radius so that the ball $B(z_t,r_t)$ is compactly contained in $S_+$. We may now choose a sequence $z_{t_1},z_{t_2},z_{t_3},\ldots$ of points from the $z_t$ such that the union of all the balls $B(z_{t_j},r_{t_j}/2)$ intersects every possible ray $L_t$, $t \in (0,1)$.

For each $j = 1,2,\ldots$, let $S_j \subset S$ denote the set defined as the union of all line segments $L_t$ which intersect the ball $B(z_{t_j},r_{t_j}/2)$, and $V_j$ the set $S_j \cap B(z_{t_j},r_{t_j})$. We now choose a smooth function $\psi_j : \cc \to \rr$ supported on $S_j$ with the following two properties.
\begin{itemize}
\item If $z \in V_j$, then $\psi_j(z) < 0$ and if $z \in S_j \setminus \bar{V_j}$, then $\psi_j(z) > 0$.
\item The integral of $\psi_j$ over each segment $L_t$ is zero.
\end{itemize}
Such a function is not difficult to construct so we omit the details. We then scale $\psi_j$ down so that both $|\psi_j|$ and $|\nabla \psi_j|$ are uniformly bounded from above  by $2^{-j}$. If necessary, we scale $\psi_j$ further down so that the inequality
\begin{equation}\label{eq:psiineq}
\partial_r H > -2^j\psi_j
\end{equation}
always holds in $\bar{B(z_{t_j},r_{t_j})}$ - this is possible since $\partial_r H \ge c>0$ on $\bar{B(z_{t_j},r_{t_j})}$ by continuity.  We then define a function $\Psi_j$ by
\[\Psi_j(r e^{i\theta}) = \int_0^r \psi_j(t e^{i\theta}) dt.\]
Hence $\partial_r \Psi_j = \psi_j$. Due to the bounds on $\psi_j$ and $\nabla \psi_j$, the sum
\[\Psi := \sum_{j=1}^\infty \Psi_j\]
converges in $\mathscr W^{1,p}(\cc)$.

We now note that the map $\Psi+H$ satisfies all the properties \ref{it:1}., \ref{it:3}. and \ref{it:4}. we required from $H$ and is also injective on each $L_t$; that is, it satisfies the property \ref{it:2}. To verify this, $\Psi$ is zero on $\partial S$ the boundary values are unchanged. Since $\Psi$ is in $\mathscr W^{1,p} (\mathbb C)$ the Sobolev-regularity is preserved. We now claim that $\partial_r(\Psi + H) > 0$. If we are in one of the sets $B(z_{t_j},r_{t_j})$, the inequality \eqref{eq:psiineq} applies. Hence we find that
\[-\partial_r \Psi = \sum_{j=1}^\infty -\psi_j < \sum_{j=1}^\infty 2^{-j} \partial_r H = \partial_r H.\]
If a point $z$ does not belong to any of the $B(z_{t_j},r_{t_j})$, we have that $\psi_j(z) \geq 0$ for all $j$. Furthermore, $z$ must belong to one of the segments $L_t$. Take a $j$ so that $L_t$ intersects $B(z_{t_j},r_{t_j}/2)$, which implies that $z$ belongs to the set $S_j \setminus \bar{V_j}$. Since $\psi_j$ is positive on this set we find that
\[\partial_r \Psi(z) + \partial_r H(z) \geq \psi_j(z) + \partial_r H(z) > 0.\]
This proves our claimed property $\partial_r(\Psi + H) > 0$, which implies that the map $\Psi + H$ is injective on $L_t$. Thus a map with the claimed properties \ref{it:1}.-\ref{it:4}.
\\\\
\emph{Step 3.} Regularizing the boundary map.\\\\
Let us again return to the case where $\yy$ consists of three curves $\Gamma_1,\Gamma_2$ and $\Gamma_3$ as before in Step 1. Our aim is to slightly deform $\Gamma_1$ inside the domain $\yy$ and use the boundary diffusion to replace our given boundary value with a smooth map. To this end, suppose via affine transformation that $\Gamma_1$ is the graph of a smooth function $\Phi$ over the interval $[0,T_0]$ on the $x$-axis, with the intersection point $P$ laying at the origin. We may also suppose that locally the domain $\yy$ is below the curve $\Gamma_1$ and the complement is above it. Since there was some freedom in choosing $\Gamma_3$, we may suppose that $\Gamma_3$ contains a small vertical segment which starts from $(T_0,\Phi(T_0))$ and ends at $(T_0, (1-\varepsilon)\Phi(T_0))$.

Let us now modify the boundary diffusion $H$ defined earlier to define a new map. Writing the boundary diffusion $H$ in coordinates as
\[H(x,y) = (\,A(x,y)\, , \, \Phi(A(x,y))\,),\]
where $\Phi$ denotes the smooth function whose graph $\Gamma_1$ is. Let $f:[0,1] \to [0,1]$ be a smooth strictly increasing function with $f(0)=0$ which will be ''small'', exactly how small we will choose later. We use $f$ to define a new map $H^*$ on the triangle $S$ as
\[H^*(x,y) := \left(\,A(x,y)\, , \, \left(1 - \frac{y}{x}f(A(x,y))\right)\Phi(A(x,y))\,\right).\]
Let us now explain the properties of $H^*$ and also how $f$ is chosen. First of all we wish to verify that $H^*$ lies in the Sobolev space $\mathscr W^{1,p}(S)$. Since $A(x,y)$ is the real part of $H$, we find that $A \in \mathscr W^{1,p}(S)$. As $f$ and $\Phi$ are smooth functions, the issue only lies in the factor $y/x$ in the definition of $H^*$ which is bounded in $S$ but does not have a bounded derivative. 

Thus to conclude that $H^* \in \mathscr W^{1,p}(S)$, we must show that the expression $\frac{y}{x^2} f(A(x,y))$ remains bounded in $S$. Since $y \leq x$ in $S$, we only require that $f(A(x,y)) \leq x$. For each $x$, we consider the quantity $\tau(x) = \max_{y \leq x} A(x,y)$. By the construction of $H$, in particular the fact that $H$ takes each of the segments $L_t$ homeomorphically onto $\Gamma_1$ and continuity, we find that $\tau$ is strictly increasing, continuous, and $\tau(0) = 0$. Let $\tau^{-1}$ denote its inverse function. Choosing now $f$ to be smooth, strictly increasing, and so small that $f(t) \leq \tau^{-1}(t)$, we find that $f(A(x,y)) \leq f(\tau(x)) \leq x$. Hence $H^* \in \mathscr W^{1,p}(S)$.

Now note that on each $L_t$, the quantity $y/x$ is the constant $t$ and hence the map $H^*$ takes each line segment $L_t$ homeomorphically onto the curve
\[\gamma_t(s) := (\,s\, , \, (1-tf(s))\Phi(s)\,).\]
These curves are the graphs of smooth functions over $[0,T_0]$ which start at the origin. Note that the smaller function $f$  is, the closer these curves are to $\Gamma_1$. Hence we may choose $f$ so small that all of these curves  lie in $\yy$. Let now $\yy_0$ denote the region bounded by $\Gamma_1$, $\gamma_1$, and the vertical line segment between $(1,\Phi(1))$ and $(1,(1-f(1))\Phi(1))$ which joins the other endpoints of the former two curves. This vertical line segment may be assumed to be a part of $\Gamma_3$. This way $H^*$ becomes a $\mathscr W^{1,p}$-Sobolev homeomorphism from $S$ to $\yy_0 \subset \yy$. Moreover, since $H$ is Lipschitz on $L_1$ we find that $H^*$ is a Lipschitz map from $L_1$ to $\gamma_1$ and it is also a linear map on the line segment from $(1,0)$ to $(1,1)$. Moreover, it is equal to the given boundary value $\varphi_1$ on $L_0$.\\\\
\emph{Step 4.} Finishing the proof.\\\\
In essence, the construction of the map $H^*$ has allowed us to replace the boundary value $\varphi_1$ by a Lipschitz boundary value   $\gamma_1$ which is part of the boundary of the slightly smaller domain $\yy \setminus \yy_0$. This allows us to assume from the beginning that $\varphi_1$ is Lipschitz. From a similar construction on $\Gamma_2$ we may further assume that $\varphi$ is Lipschitz on the whole boundary of $\dd$. But then it admits a homeomorphic Lipschitz-extension by Theorem~\ref{thm:lip}. Hence we have shown that $\varphi: \partial \dd \to \partial \yy$ admits a homeomorphic extension $h: \bar{\dd} \to \bar{\yy}$ in $W^{1,p}(\dd)$.

To extend the map $h$ into the complement $\cc \setminus \dd$ is now quite simple. For example we may use the following reflection argument. Suppose that $0 \in \dd$ and $0 \in \yy$. Let $\tau(z) = 1/\bar{z}$ denote an inversion map, and define $\yy^* = \tau(\yy)$ so that $\yy^*$ is also a piecewise smooth domain. Consider the boundary map $\tau \circ \varphi : \partial \dd \to \partial \yy^*$. Since $\tau$ is locally bilipschitz in $\cc \setminus \{0\}$, $\tau \circ \varphi$ satisfies the $p$-Douglas condition. It thus admits a homeomorphic extension $h^* : \bar{\dd} \to \bar{\yy^*}$ in $W^{1,p}(\dd)$. Now if $z \in \cc \setminus \dd$, we define $h(z) = \tau(h^*(\tau(z)))$. Again using the fact that $\tau$ is locally bilipschitz in $\cc \setminus \{0\}$, the identity on $\partial \dd$ and an involution, we readily see that this defines $h: \cc \to \cc$ as a homeomorphism in $W^{1,p}_{\loc}(\cc,\cc)$ and equal to $\varphi$ on $\partial \dd$.
\end{proof}

\begin{proof}[Proof of Theorem \ref{thm:p=2}] The proof of this theorem is simply a repeat of the above proof of Theorem \ref{thm:mainp>2}. The only part where the assumption $p > 2$ was used in that proof was in the proof of Lemma \ref{lem21}. However, it is clear that the assumption \eqref{eq:modcontphi} implies the statement of Lemma \ref{lem21} in our case. Thus the proof is complete.
\end{proof}

\section{The general extension result}\label{sec:genext}

In this section we prove the following general extension result, which will be used in the proofs of Theorem~\ref{thm:quasidisk} and Theorem~\ref{thm:mainonesided}. This result may also be of independent interest as we expect it could be applied to future studies as well. Before stating the theorem, we describe the notion of a dyadic family of arcs on $\partial \dd$.

For a fixed $n_0 \in \mathbb{N}$, a family of closed arcs $I = \{I_{n,j} \subset \partial \dd : n \geq n_0, \ j = 1,2,\ldots,2^n\}$ is called dyadic if the following conditions hold. For each fixed $n$ there are $2^n$ arcs $I_{n,j}$ in $I$ which are of equal length, pairwise disjoint apart from their endpoints, and cover $\partial \dd$. For each arc $I_{n,j}$ there are two arcs in $I$ of half the length of $I_{n,j}$ and so that their union is exactly $I_{n,j}$, these intervals are called the \emph{children} of $I_{n,j}$ and $I_{n,j}$ is the \emph{parent}. 

\begin{theorem}\label{thm:generalext} Suppose that $\yy$ is a Jordan domain and $\varphi : \partial \dd \to \partial \yy$ is a boundary homeomorphism. For $n_0 \in \mathbb{N}$ suppose that there is a dyadic family $I = (I_{n,j})$ of closed arcs on $\partial \dd$ such that the following hold.
\begin{itemize}
\item For each $I_{n,j}$, there exists a crosscut $\Gamma_{n,j}$ in $\yy$ connecting the two endpoints of the boundary arc $\varphi(I_{n,j}) \subset \partial \yy$ and such that the estimate
\begin{equation}\label{gammasumcond}\sum_{n=n_0}^\infty 2^{(p-2)n}\sum_{j=1}^{2^n} |\Gamma_{n,j}|^p \, < \, \infty\end{equation}
holds.
\item The crosscuts $\Gamma_{n,j}$ are all pairwise disjoint apart from their endpoints at the boundary, where $n,j$ are allowed to range over all their possible values.
\end{itemize}
Then $\varphi$ admits a homeomorphic extension from $\bar{\dd}$ to $\bar{\yy}$ in the class $\mathscr W^{1,p}(\dd)$.
\end{theorem}
\begin{proof} Note first that since the curves $\Gamma_{n,j}$ are constructed based on the dyadic family $I$, the notion of children and parents is inherited from $I$ to these curves.
The outline of the proof is  simple. We a make a  dyadic type decomposition of $\dd$ with sets $U_{n,j}$. Accordingly we split our target domain $\yy$ into sets $V_{n,j}$ which correspond to the image of  $U_{n,j}$. The length of the boundary of $V_{n,j}$ is controlled by the length of $\Gamma_{n,j}$ and its children.  We then obtain the desired extension by constructing a Lipschitz homeomorphism from each $U_{n,j}$ onto $V_{n,j}$, and control the Sobolev-norm by the sum in \eqref{gammasumcond}. The proof is split into a number of steps as follows.
\\\\
\emph{Step 1.} Reducing the problem to a triangle.\\\\
We first consider the highest generation of curves $\Gamma_{n_0,j}$, $j = 1,\ldots,2^{n_0}$. These curves split the domain $\yy$ into one central domain $\yy_0$ and $2^{n_0}$ domains $\yy_j$, $j = 1,\ldots,2^{n_0}$, which are bounded by the $\Gamma_{n_0,j}$ and the boundary arcs $\varphi(I_{n,j})$. On $\dd$ we make a similar construction, connecting the endpoints of the arcs $I_{n,j}$ via smooth, pairwise disjoint curves in $\dd$ so that $\dd$ also splits into a central domain $D_0$ and $2^{n_0}$ domains $D_j$. We choose these smooth curves to form positive angles with $\partial \dd$ and each other so that all of these domains $D_j$, $j = 0,\ldots,2^{n_0}$ are piecewise smooth with no angles of size zero on the boundary.

We wish to define the homeomorphic extension $h : \dd \to \yy$ of $\varphi$ so that it sends the smooth curves bounding the $D_j$ inside $\dd$ at constant speed to the curves $\Gamma_{n_0,j}$. If this is done, then it is enough to define how $h$ maps $D_j$ to $\yy_j$ for all $j = 0,\ldots,2^{n_0}$. For the central domain case $j = 0$, since $D_0$ is bilipschitz-equivalent with the unit disk we may simply use a Lipschitz homeomorphic extension from $D_0$ onto $\yy_0$ as given by Theorem~\ref{thm:lip}. 

To finish our reduction, let $T$ denote the isosceles triangle in the plane with vertices at $(-1,0)$, $(1,0)$ and $(0,1)$. Since each of the sets $D_j$ is bilipschitz-equivalent with $T$, we replace the sets $D_j$ with $T$ in our construction. Since the construction is similar for each $j$, we suppose that $j=1$. We may also suppose that our given boundary map, still denoted by $\varphi$ but now defined on $\partial T$ with some abuse in notation, maps the base of the isosceles $\partial T$ on the real line to $\partial \yy$ and the legs onto the curve $\Gamma_{n_0,1}$ with constant speed. The arcs $I_{n,j}$ can now be supposed to lie on the real line and form a dyadic decomposition of the base $[-1,1]$ of $T$. This can be done by choosing the bilipschitz map from $D_1$ to $T$ to map the boundary arc $I_{n_0,1}$ to $[-1,1]$ at constant speed. It remains to see how the extension $h$ is defined from $T$ to $\yy_1$.
\\\\
\emph{Step 2.} Defining the decomposition of $T$ and $\yy_1$.\\\\
\begin{figure}
\includegraphics[scale=0.5]{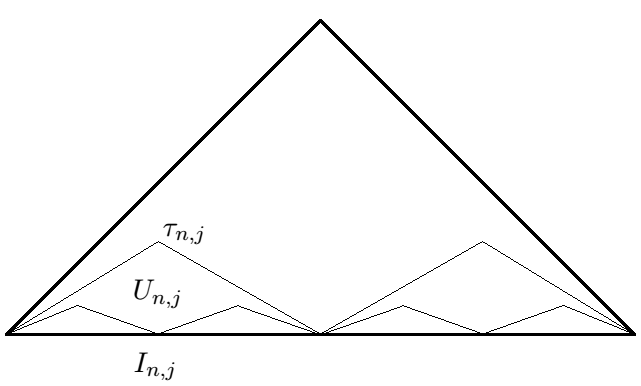}
\caption{Splitting the triangle $T$ into parts.}\label{triangles1}
\end{figure}
We first define sets $U_{n,j} \subset T$ as follows. Let $(\alpha_n)$ denote a fixed strictly decreasing sequence of angles so that $\alpha_{n_0} = \pi/4$ and $\alpha_n \to 0$ as $n \to \infty$. For each interval $I_{n,j}$ we define a curve $\tau_{n,j}$ over this interval, called the \emph{legs} of $I_{n,j}$, by letting $\tau_{n,j}$ consist of the legs of an isosceles triangle with base $I_{n,j}$ and base angles equal to $\alpha_n$. Thus the legs of the largest dyadic interval, which is simply $[-1,1]$, are the legs of the original triangle $T$. Now for each interval $I_{n,j}$, the set $U_{n,j}$ is the set bounded by the legs $\tau_{n,j}$ and the legs of the two children of $I_{n,j}$. We denote the children of $I_{n,j}$ by $I_{n,j}^-$ and $I_{n,j}^+$ from left to right and their legs by $\tau_{n,j}^-$ and $\tau_{n,j}^+$ respectively. The idea is that the extension $h$ will map the legs $\tau_{n,j}$ onto the curve $\Gamma_{n,j}$.

We define sets $V_{n,j} \subset \yy$ analogously  to the sets $U_{n,j}$. For each interval $I_{n,j}$, we consider the corresponding curve $\Gamma_{n,j}$ and its two children, denoted by $\Gamma_{n,j}^-$ and $\Gamma_{n,j}^+$. Then the set $V_{n,j}$ is defined as the set bounded by these three curves in $\yy$, see Figure~\ref{gammafig1}.
\\\\
\emph{Step 3.} Defining the map $h$ from $U_{n,j}$ to $V_{n,j}$.\\\\
\begin{figure}
\includegraphics[scale=0.5]{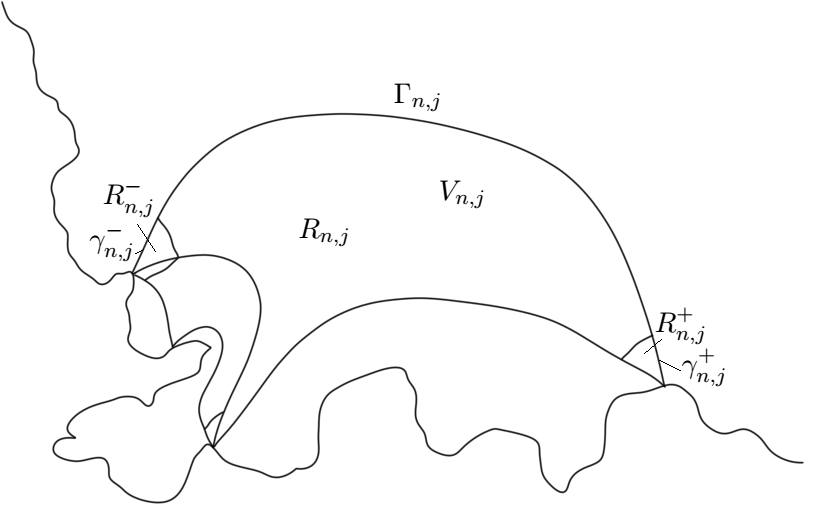}
\caption{The construction of the curves inside $\yy$.}\label{gammafig1}
\end{figure}
We would like to simply take a Lipschitz homeomorphism from $U_{n,j}$ to $V_{n,j}$, but since the sets $U_{n,j}$ have angles that tend to zero as $n \to \infty$ they are not uniformly bilipschitz-equivalent to a scaled down copy of the unit disk. Therefore we must find another way to construct $h$ on $U_{n,j}$. We do this by first splitting $U_{n,j}$ into three smaller sets. Note that the midpoint of the interval $I_{n,j}$ is the endpoint of two line segments on the boundary of $U_{n,j}$. We extend these segments to the opposite direction from this midpoint until they meet the legs $\tau_{n,j}$. This splits the set $U_{n,j}$ into one central quadrilateral $S_{n,j}$ and two thin triangles $S_{n,j}^-$ and $S_{n,j}^+$ with one angle equal to $\beta_n = \alpha_n - \alpha_{n+1}$.
\begin{figure}
\includegraphics[scale=0.5]{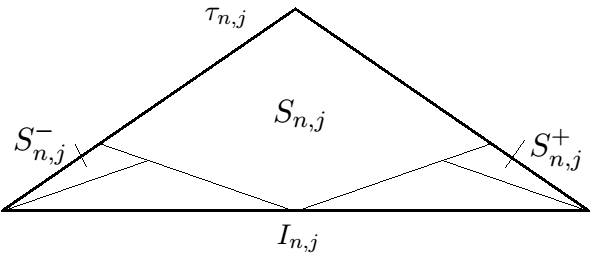}
\caption{The sets $S_{n,j}$, $S_{n,j}^-$ and $S_{n,j}^+$ defined.}\label{rectfig}
\end{figure}

We must split the set $V_{n,j}$ similarly into three regions. To facilitate this, we will let $d_n$ denote a very small distance to be defined in a moment. For each curve $\Gamma_{n,j}$ we isolate two parts of length $d_n$ from this curve which start from the two endpoints $\Gamma_{n,j}$, called $\gamma_{n,j}^-$ and $\gamma_{n,j}^+$ respectively and so that the $\pm$-signs match the children of $\Gamma_{n,j}$. Let us concentrate on the endpoint of $\Gamma_{n,j}$ where $\gamma_{n,j}^-$ starts from. We denote the part of $\Gamma_{n,j}^-$ (the child of $\Gamma_{n,j}$) which starts from this endpoint and has length $d_{n+1}$ by $\gamma_{n,j}^{--}$. We connect the other endpoints of $\gamma_{n,j}^{-}$ and $\gamma_{n,j}^{--}$ by a curve of length at most $2(d_n + d_{n+1})$ inside $V_{n,j}$ to separate a small ``triangle'' $R_{n,j}^-$ from $V_{n,j}$. The existence of such a curve is obvious as we may simply slightly deform the union of the curves $\gamma_{n,j}^{-}$ and $\gamma_{n,j}^{--}$ to obtain one. We similarly define $R_{n,j}^+$, and the remaining part of $V_{n,j}$ is denoted by $R_{n,j}$,  see Figure~\ref{gammafig1}.

We now construct a homeomorphism from $S_{n,j}^-$ to $R_{n,j}^-$ as follows. The numbers $d_n$ will be chosen here. First of all, we must have that for each $n$, $d_n < \min_j |\Gamma_{n,j}|/2$ or otherwise the construction  does not make sense. Second, we will require that the sequence $(d_n)$ is decreasing, which means that the perimeter of $R_{n,j}^-$ is comparable to $d_n$. We will first construct the map $h$ as a homeomorphism from $\partial S_{n,j}^-$ to $\partial R_{n,j}^-$ so that it takes each side of the triangle $S_{n,j}^-$ at constant speed to one of the curves $\gamma_{n,j}^{-}$, $\gamma_{n,j}^{--}$, and the third curve that bounds $R_{n,j}^-$. The order of which side goes to which curve is determined since we require that $h$ maps each $\tau_{n,j}$ to $\Gamma_{n,j}$. We must now choose $d_n$ so small that $h$ extends to a Lipschitz homeomorphism from $S_{n,j}^-$ to $R_{n,j}^-$ with Lipschitz constant at most $4^{-n}$. This can be done since we can first map $S_{n,j}^-$ to the unit disk via a bilipschitz map, and although the Lipschitz constant of this map blows up when $n \to \infty$ we may choose $d_n$ small enough so that when composed with the Lipschitz homeomorphic extension from Theorem~\ref{thm:lip} the Lipschitz constant of the combined map is less than $4^{-n}$.

At this point we may already note that the Sobolev-norm of $h$ over all the sets $S_{n,j}^-$ can be estimated from above by
\[\int_{\bigcup_{n,j} S_{n,j}^-} |Dh|^p \, dx \, dy \, \leq \sum_{n=n_0}^\infty \sum_{j=1}^{2^n} |S_{n,j}^-| 4^{-pn} \leq  \sum_{n=n_0}^\infty 2^n 4^{-pn} < \infty,\]
where we have simply observed that $|S_{n,j}^-| \leq 1$. We similarly map each $S_{n,j}^+$ to $R_{n,j}^+$ with an analogous estimate on these sets.

It remains to see how the central quadrilateral $S_{n,j} \subset U_{n,j}$ is mapped to $R_{n,j}$. Again, we first define $h$ on $\partial S_{n,j}$. The map $h$ is already defined on four small parts of $\partial S_{n,j}$. Two of these are the common parts with $S_{n,j}^-$ and $S_{n,j}^+$. The other two lie on the line segments $\tau_{n,j}^-$ and $\tau_{n,j}^+$ which both end at the midpoint of $I_{n,j}$. These latter ones get mapped to curves of length $d_{n+1}$ which are part of the two children of $\Gamma_{n,j}$. Nevertheless, the Lipschitz constant on each of these four parts is less than $4^{-n}$ so they will not play a role in our estimates. On the rest of the boundary, which is now made up of three connected parts, we define $h$ at constant speed. This defines $h$ uniquely as we require that $h$ maps each $\tau_{n,j}$ to the curve $\Gamma_{n,j}$. Since the length of the set $\partial S_{n,j}$ is comparable to $2^{-n}$, the Lipschitz constant of $h$ on $\partial S_{n,j}$ is thus controlled by $2^n |\partial V_{n,j}|$, where $|\partial V_{n,j}| = |\Gamma_{n,j}| + |\Gamma_{n,j}^-| + |\Gamma_{n,j}^+|$ is the perimeter of $V_{n,j}$. Note that the set $S_{n,j}$ is uniformly bilipschitz-equivalent to a disk of radius $2^{-n}$. Hence we may scale up by $2^n$ and again apply Theorem~\ref{thm:lip} to find a homeomorphic Lipschitz extension $h : S_{n,j} \to R_{n,j}$ with Lipschitz constant at most $C 2^n |\partial V_{n,j}|$.

In total, this gives the estimate
\[\int_{\bigcup_{n,j} S_{n,j}} |Dh|^p \, dx \, dy \, \leq C\left(1+\sum_{n=n_0}^\infty \sum_{j=1}^{2^n} 2^{-2n} \left(2^n |\partial V_{n,j}|\right)^p\right).\]
The sum on the right hand side above is finite due to our assumption \eqref{gammasumcond}. Thus $h$ lies in $\mathscr W^{1,p}$ and defined in this way, $h : \bar{\dd} \to \bar{\yy}$ is  also a homeomorphism which agrees with $\varphi$ on the boundary. Thus the proof is complete.
\end{proof}

\section{The counterexamples}

In this section we prove Theorem \ref{ex:zhangp} and provide explanation for Example \ref{ex:cuspp=2}.

\begin{proof} We are to construct a Jordan domain $\yy$ and a boundary homeomorphism $\varphi: \partial\dd \to\partial\yy$ which admits a Sobolev extension but not a homeomorphic one. We first give a short outline of the construction.

The basic idea is that the boundary of $\yy$ will contain two snowflake-like arcs of infinite length. These two arcs will get closer to each other towards their common endpoint, so that this endpoint cannot be approached by a  curve with finite length from the inside of $\yy$. We then define a boundary map from $\partial \dd$ which sends a lot of mass to the endpoint. We will exploit the fact that both of these arcs are parts of a quasicircle to guarantee that the boundary map has a Sobolev extension. On the other hand since a large amount of mass is sent to a boundary point  it is difficult  to approach from the inside of $\Y$. This will show that a homeomorphic extension has infinite $\W^{1,1}$-energy.
\\\\
Let us first discuss snowflakes. The typical Koch-type snowflake curve is constructed as follows. We choose a parameter $\tau \in (1/4,1/2)$, take the unit interval $I_0 = [0,1]$ and replace this interval by four line segments of length $\tau$ as in Figure \ref{snowfig} to obtain a new curve $I_1$. We then continue this process, replacing each segment in $I_n$ by the configuration in Figure \ref{snowfig} appropriately scaled, translated and rotated to match the endpoints of the segment. The limit curve of this process is called a Koch-type snowflake and it always has infinite length.

\begin{figure}
\includegraphics[scale=0.9]{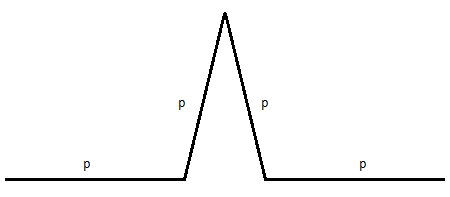}
\caption{The base curve used to construct a snowflake.}\label{snowfig}
\end{figure}

For a number $\epsilon > 0$ and fixed $n\in \mathbb{N}$, we define an $\epsilon,n$-snowflake tube $I_{n,\epsilon}$, which is a piecewise linear Jordan curve, as follows. We first take the curve $I_n$ from above and remove one quarter of its segments by removing exactly those segments that were constructed based on the rightmost segment of $I_1$. This gives a new curve $I_n'$ and the tube $I_{n,\epsilon}$ is defined by constructing the following curve around $I_n'$. For each line segment $S$ in $I_n'$, we take a rectangle with the same width as $S$ and height $2\epsilon$, placed so that the edges of length $2\epsilon$ are perpendicular to $S$ and contain the endpoints of $S$ as their midpoints. In all cases the parameter $\epsilon$ may be chosen as small as we wish and it will be chosen small enough that the height of each rectangle  is always less than $1/10$ of the width. For each two rectangles $R_1$ and $R_2$ based on two neighboring segments, we join the two ends of these rectangles which intersect as in Figure \ref{piecesfig}. For a more formal description of this conjoining process: First remove the two intersecting shorter sides from each of the two rectangles. Then remove the shorter ends of the two longer sides that intersect. Finally extend the two longer sides which do not intersect until they do. See also Figure \ref{tubefig} for an example of the end result. For the first and last segment of $I_n'$, going from left to right, the rectangle constructed on that segment has two edges that did not need to be joined with another rectangle. We call these two segments the left- and right end-edges of $I_{n,\epsilon}$ respectively.

\begin{figure}
\includegraphics[scale=0.7]{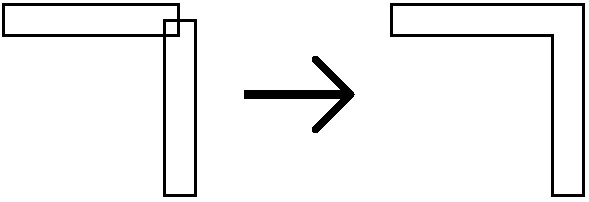}
\caption{Joining two neighboring  rectangles.}\label{piecesfig}
\end{figure}

\begin{figure}
\includegraphics[scale=0.9]{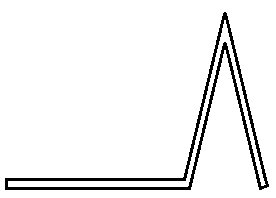}
\caption{An example of a $\epsilon,n$-snowflake tube.}\label{tubefig}
\end{figure}

We now define a way to join these types of tubes to create the boundary of a Jordan domain $\yy$. Given a rapidly decreasing sequence $\epsilon_1, \epsilon_2, \ldots$ of positive numbers and a sequence of positive integers $n_1,n_2,\ldots$ we define a curve $\Gamma$ as follows. We first consider the curve $I_{n_1,\epsilon_1}$. We take the curve $I_{n_2,\epsilon_2}$, scale it by a factor of $1/4$ to create a new curve $I'$ and translate it so that the right end-edge of $I_{n_1,\epsilon_1}$ and the left end-edge of the scaled curve $I'$ have the same midpoint. We then modify the curve $I_{n_1,\epsilon_1}$ slightly by scaling and rotating its right end-edge so that this edge will be equal to the left end-edge of $I'$. We have essentially joined the first tube with the second one, scaled down, see Figure \ref{tubefig2}. We continue this process, adding each tube $I_{n_j,\epsilon_j}$ to the previous one by scaling it down to $1/4$ of the previous one, translating it appropriately and modifying the right end-edge of the previous one to align with the left end-edge of the scaled and translated copy of $I_{n_j,\epsilon_j}$. After scaling down each tube $I_{n_j,\epsilon_j}$, translating it and modifying the right end-edge so that the next tube is able to be joined, the area surrounded by this modified tube will be denoted by $Y_j$ and called the $j$:th part of $\yy$. Continuing this process infinitely makes this process converge to a Jordan domain $\yy$. The boundary $\partial \yy$ can be thought of as being an union of three parts: the leftmost vertical line segment $L_1$ which is the left end-edge of $I_{n_1,\epsilon_1}$ and the two parts $L_2$ and $L_3$ which the remaining boundary splits into when we divide it at the rightmost point of $\partial \yy$.

The specific choice of the sequences $(\epsilon_j)$ and $(n_j)$ above is not important as long as the following conditions are satisfied. We choose each $n_j$ so large and $\epsilon_j$ so small that to traverse from the left end-edge of $Y_j$ to its right end-edge the minimal path length is at least $4^j$.

\begin{figure}
\includegraphics[scale=0.9]{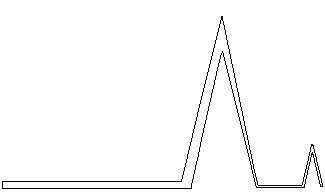}
\caption{Joining two tubes.}\label{tubefig2}
\end{figure}

We now define the boundary map $\varphi: \partial \dd \to \partial \yy$. First, we split $\partial \dd$ into three arcs $A_1,A_2$ and $A_3$, say each of length $2\pi/3$. Under $\varphi$, we map each $A_j$ to the part $L_j$ of $\partial \yy$ as follows. First, we map $A_1$ at constant speed to the segment $L_1$. For $j = 2,3$ the construction is the same in both cases, so we will just explain how $A_2$ is mapped to $L_2$. We divide $A_2$ into disjoint arcs $a_1,a_2,\ldots$ so that $|a_j| = 4|a_{j+1}|$ in the most natural way, meaning that one of the endpoints of $a_1$ is also an endpoint of $A_1$ and the arcs $a_j$ and $a_{j+1}$ always share an endpoint. See Figure  \ref{tubefig3} for a rough illustration of the whole picture.

\begin{figure}
\includegraphics[scale=0.4]{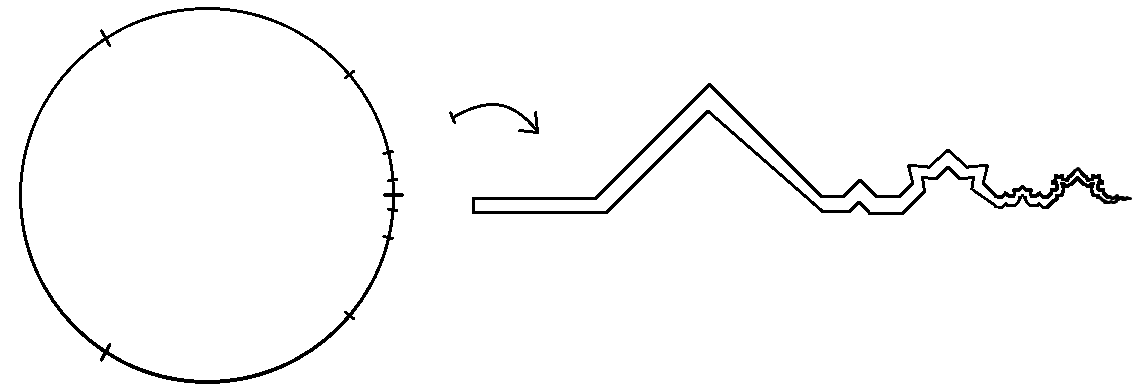}
\caption{A rough illustration of how $\partial \dd$ is divided and mapped to the target domain $\yy$.}\label{tubefig3}
\end{figure}

We map each arc $a_j$ to the corresponding part $\ell_j = Y_j \cap L_2$ on $L_2$ by noting that $\ell_j$ consists of $4^{n_j}$ segments of equal length except possibly for one segment at the rightmost end which could have a slightly different length due to the fact that we had to modify it to glue $Y_j$ to the next part $Y_{j+1}$, however the length of that segment can be supposed to be comparable to the others by a constant independent of $j$. Hence we split the arc $a_j$ also into $4^{n_j}$ equal length smaller arcs and map these each at constant speed to segments in $\ell_j$. This completely defines the map $\varphi$ from $\partial \dd$ to $\partial \yy$.

Note now that both curves $L_2$ and $L_3$ are arcs of a quasicircle since even though they are not exactly the Koch snowflake curve, they can both be associated to a larger class of snowflake curves as defined by Rohde in \cite{Ro}.   In particular,  they are bilipschitz equivalent to such a snowflake curve since there are some slight deformations from the gluing process for example. Moreover, the boundary map $\varphi$ from each $A_j$ to $L_j$ is quasisymmetric and H\"older-continuous, the verification of this takes some more work but it is done in \cite{KKO}, proof of Theorem 1.6. The H\"older-exponent of $\varphi$ can be chosen to be arbitrarily close to $1$ by making the parameter $\tau$ smaller. It is easy then  to verify  that for each $p > 1$, there is $\tau > 1/4$ so that $\varphi$ admits a $\mathscr W^{1,p}$-Sobolev extension since the $p$-Douglas condition is always satisfied by a  H\"older-continuous mapping with H\"older-exponent close enough to $1$. It remains to verify that no homeomorphic extension exists.

Suppose that $\varphi$ would admit a $\W^{1,1}$-homeomorphic extension $h : \dd \to \yy$. We split the arc $A_1 \subset \partial \dd$ into sub-arcs $b_1,b_2,\ldots$ in counterclockwise order so that $|b_j| = 4|b_{j+1}|$. Suppose that the arc $A_2$ is the neighbor  of $A_1$ which is before $A_1$ in counterclockwise direction, as in Figure \ref{tubefig3}. For each $j \geq 2$, we then connect each arc $b_j$ with the arc $a_j \subset A_2$ from before by straight line segments between each pair of points in them. The union of these line segments will be called $R_j$, and the sets $R_j$ have disjoint interiors. See Figure \ref{r2fig}. We may abuse notation and think of each $R_j$ as a rectangle of width $1$ and height $4^{-j}$, since they are bilipschitz-equivalent to such a rectangle with a bilipschitz constant independent of $j$.

\begin{figure}
\includegraphics[scale=0.5]{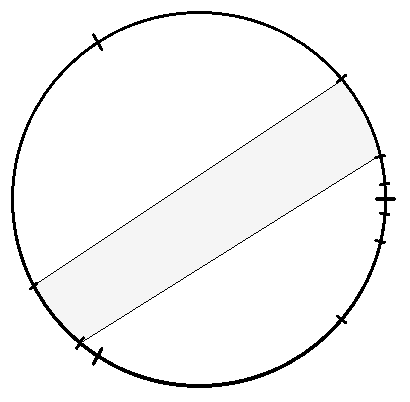}
\caption{The ''rectangle'' $R_2$.}\label{r2fig}
\end{figure}

The smaller, vertical sides of each rectangle $R_j$ are thought to be part of $\partial \dd$, with one side in $A_1$ mapped to the line segment $L_1$ and one side $a_j \subset A_2$ mapped to $\ell_j \subset L_2$. Each of the horizontal line segments between these two vertical sides is mapped to a curve in $\bar{\yy}$ which starts at $L_1$ and ends at $\ell_j$. By construction, such a curve has length comparable to at least $4^j$ by a uniform constant $C$. Thus
\[\int_{R_j} |Dh|\, dx \, dy = \int_0^{\frac{1}{4^j}}\int_0^1 |Dh|\, dx \, dy \geq C\int_0^{\frac{1}{4^j}} 4^j \, dy = C.\]
Since the energy over the sets $R_j$ is not summable in $j$, this contradicts the assumption $h \in \W^{1,1}(\dd)$ and completes the proof.
\end{proof}

Next, we provide explanation for Example \ref{ex:cuspp=2}.

\begin{proof}
Using a conformal map, the result of this example was already shown in \cite{KOext} but with some restrictions on the range of $\beta$ (the polynomial degree of the cusp). Here we will provide a more direct construction and use the same reasoning to provide an example for every $\beta > 1$. \\\\
It is known that a homeomorphism of $\W^{1,2}$-regularity cannot have a local modulus of continuity worse than $\omega(x) = \log^{-1/2}(e + 1/x)$. Hence to find a boundary homeomorphism which satisfies the Douglas condition but cannot be extended as a $\W^{1,2}$-homeomorphism to the inner cusp region $\Omega_\beta$, it would be enough to find a boundary map with a modulus of continuity worse than $\omega(x)$ above but which can be extended as a $\W^{1,2}$-homeomorphism to the complement of $\Omega_\beta$. Since the domain $\Omega_\beta$ is smooth apart from the cusp of degree $\beta$, we may restrict our considerations to a neighbourhood of this cusp. Hence it would be enough to find a boundary map from some Lipschitz domain to an region with an outer cusp of degree $\beta$ so that the boundary map has a modulus of continuity worse than $\omega$ and admits an extension as a $\W^{1,2}$-homeomorphism to this outer cusp region.
\\\\
Let $f(x) = c\log^{-\epsilon}(x)$, where $\epsilon < 1/2$ is chosen so that $\beta\epsilon > 1/2$ and the constant $c$ is such that $f(1) = 1$. Let $\xx = {(x,y) \in \rr^2 : 0 \leq x \leq 1, |y| \leq x}$ be a triangular region and $\xx = {(x,y) \in \rr^2 : 0 \leq x \leq 1, |y| \leq x^\beta}$ be a cusp. We define a map $h : \xx \to \yy$ by
\[h = (u,v) = \left(f(x),f(x)^\beta \frac{y}{x} \right),\]
so that $h$ maps vertical segments to vertical segments and its modulus of continuity is comparable to $f$ on the boundary. We wish to verify that $h$ belongs to $\W^{1,2}(\xx)$. The relevant derivatives are
\[u_x  = f'(x), \quad v_x = \beta f'(x) f(x)^{\beta - 1} \frac{y}{x} \quad \text{ and } \quad v_y = \frac{f(x)^\beta}{x}.\]
Since both $f(x)^{\beta - 1} \leq 1$ and $\frac{y}{x} \leq 1$, we find that $|v_x| \leq |u_x|$ so we may omit $v_x$ from our considerations. For the rest, we compute that
\begin{align*}
\int_\xx |Dh|^2 \, dy \, dx &= \int_0^1 \left(\int_{-x}^x dy\right) \left(f'(x)^2 + \frac{f(x)^{2\beta}}{x^2}\right) dx
\\&= 2c^2 \int_0^1 \left(\frac{1}{x \log^{2+2\epsilon}(x)}+ \frac{1}{x \log^{2\beta\epsilon}(x)}\right) dx.
\end{align*}
This integral is finite since both $2 + 2\epsilon > 1$ and $2\beta\epsilon > 1$. Since $h$ has the required modulus of continuity on the boundary, the proof is complete.
\end{proof}

\section{Extending to a quasidisk}

In this section we prove Theorem \ref{thm:quasidisk}.

\begin{proof} We wish to simply apply Theorem \ref{thm:generalext} to obtain an extension of $\varphi$ to $\bar{\dd}$. After this is done, it will be trivial to extend the map to $\cc \setminus \dd$ as well, for example by reflecting it via bilipschitz map on both the domain and the target side using the reflection property for quasidisks - see for example Chapter II.2 in \cite{GeMo}. To verify the conditions required by Theorem \ref{thm:generalext}, we let $I = (I_{n,j})$ denote a collection of dyadic arcs on $\partial \dd$. To recall briefly what this means, for each fixed $n$ there are $2^n$ arcs in $I$ which are of equal length and cover $\partial \dd$ disjointly apart from their endpoint. For each arc $I_{n,j}$ there are two arcs in $I$ of half the length of $I_{n,j}$ and so that their union is exactly $I_{n,j}$, called the \emph{children} of $I_{n,j}$. We also call the \emph{offspring} of such an interval the collection of its children, the children of its children, and so forth.

We define a collection $\Gamma = (\Gamma_{n,j})$ of curves in $\yy$ by defining $\Gamma_{n,j}$ as the geodesic w.r.t. the hyperbolic metric in $\yy$ which connects the two endpoints of the boundary arc $\varphi(I_{n,j}) \subset \partial \yy$. The notion of children and offspring is inherited from $I$ to $\Gamma$. It remains to verify that the curves in $\Gamma$ are pairwise disjoint and that
\[\sum_{n=1}^\infty 2^{(p-2)n}\sum_{j=1}^{2^n} |\Gamma_{n,j}|^p   \, < \, \infty.\]
Verifying that these hyperbolic geodesics are pairwise disjoint is not too difficult. Suppose first that two such curves $\Gamma_1$ and $\Gamma_2$ share exactly one endpoint. Then we consider the unique conformal map $g : \dd \to \yy$ which, for each endpoint  $P$ of $\Gamma_1$ or $\Gamma_2$, maps $\varphi^{-1}(P)$ to $P$. The preimages of $\Gamma_1$ and $\Gamma_2$ are also hyperbolic geodesics in $\dd$ since $g$ is conformal, and since these preimages are disjoint in $\dd$ also $\Gamma_1$ and $\Gamma_2$ must be disjoint. An immediate consequence of this observation is that each $\Gamma_{n,j}$ is disjoint with its two children. Since the children of each $\Gamma_{n,j}$ must then belong to the region in $\yy$ bounded by $\Gamma_{n,j}$ and $\varphi(I_{n,j})$, we find by an easy induction argument that the offspring of each $\Gamma_{n,j}$ must also be pairwise disjoint and disjoint with $\Gamma_{n,j}$. We also observe that for each pair of curves that are disjoint and neither is the offspring of the other, both of their offspring must then also be disjoint with the others'. It hence only remains to check what happens for small values of $n$.

For $n=1$ there are only two intervals $I_{1,1}$ and $I_{1,2}$, and their images under $\varphi$ share two endpoints so $\Gamma_{1,1} = \Gamma_{1,2}$. The claim is hence false in this case but this is obviously only a technicality, as we may forget about the case $n=1$ in a moment. It is worth to note that $\Gamma_{1,1}$ splits the domain $\yy$ into two domains, each of which must contain two of the curves $\Gamma_{2,j}$, $j = 1,2,3,4$. These curves are disjoint from $\Gamma_{1,1}$ since each of them shares one endpoint with it, and the two which are in the same component are also disjoint since they have a common endpoint. Hence they are all mutually disjoint, which proves the claim.
\\\\
Let us now aim for the desired estimate on the lengths of $|\Gamma_{n,j}|$. For each dyadic interval $I_{n,j} \in I$, we let $I_-$ and $I_+$ denote its neighbors, i.e. the two arcs in $I$ of same length as $I_{n,j}$ which share an endpoint with $I_{n,j}$. Let $d$ denote the Euclidean distance between the endpoints of $\Gamma_{n,j}$. If $x \in I_-$ and $y \in I_+$, then by the three-point property~\eqref{eq:threept} of the quasicircle $\partial \yy$ we have that $2^{-n}|x-y| \leq 3\cdot 2^{-n}$ and $|\varphi(x) - \varphi(y)| \geq C_0 d$, where $C_0$ is some uniform constant. We then recall that since $\yy$ is a quasidisk, it satisfies the hyperbolic segment property, see Chapter II.4 in \cite{GeMo}, implying that $|\Gamma_{n,j}| \leq C_1 d$. Hence we find the estimate
\[|\Gamma_{n,j}|^p\,  \leq\,  C_1 d^p \, \leq \, C 2^{n(2-p)}\int_{I_+}\int_{I_-} \frac{|\varphi(x) - \varphi(y)|^p}{|x-y|^p} \, dx \, dy.\]
Since the sets $I_- \times I_+$ are all pairwise disjoint when $n$ and $j$ range over all their possible values with $n \geq 3$, we may sum this up over all such values to find that.
\[\sum_{n=3}^\infty 2^{n(p-2)}\sum_{j=1}^{2^n} |\Gamma_{n,j}|^p \, \leq \, C\int_{\partial \dd}\int_{\partial \dd} \frac{|\varphi(x) - \varphi(y)|^p}{|x-y|^p} \, dx \, dy \, < \, \infty.\]
This proves the claim and hence Theorem \ref{thm:generalext} now gives the result.

\end{proof}

\section{Extending to a quasiconvex domain}
In this section we prove Theorem \ref{thm:mainonesided}.
\begin{proof}
Since we are dealing with a piecewise smooth boundary, we apply the same reduction  as in Step 1. of  the proof of Theorem \ref{thm:mainp>2}. Therefore it is enough to consider the target as a neighborhood of two smooth pieces at their intersection point. Hence we suppose that $\partial \yy$ consists of three curves $\Gamma_1,\Gamma_2$ and $\Gamma_3$. These are all graphs of smooth functions in some coordinate system, with $\Gamma_1$ and $\Gamma_2$ intersecting at a point $P$ and $\Gamma_3$ intersecting these two curves at a positive angle. The curves $\Gamma_1$ and $\Gamma_2$ may be assumed to intersect at an angle of zero towards the domain $\yy$, because if they met at a full angle this would contradict the quasiconvexity of $\yy$ and if they met at any other angle we could apply a bilipschitz map locally to straighten the angle out.

We may also suppose that the domain of definition is the triangle $T$ with vertices at points $(-1,0)$, $(1,0)$ and $(0,1)$ in the plane. The boundary map $\varphi : \partial T \to \partial \yy$ is assumed to satisfy the Douglas condition and we suppose that $\varphi$ takes the point $(0,0)$ to the intersection point $P$ of $\Gamma_1$ and $\Gamma_2$, and takes  $(1,0)$ and  $(-1,0)$ to the other endpoints of these curves respectively. On the preimage of $\Gamma_3$, which consists of the two non-horizontal sides of $T$, we may suppose that $\varphi$ is defined as a constant speed map onto $\Gamma_3$.

Let us partition the lower boundary of $T$ by defining the intervals $I_j^- = [-2^{-j},-2^{-j-1}]$ and $I_j^+ = [2^{-j-1},2^{-j}]$ for $j \geq 0$. For each $j$, we define a domain  $U_j$ as follows. We first let $\tau_0$ denote the two upper sides of the triangle $T$. Then with the interval $I_0^-$ as the base, we construct an isosceles triangle within the interior of $T$ and call the two legs of this triangle $\tau_0^-$. Similarly we construct $\tau_0^+$ over $I_0^+$, we may even choose $\tau_0^+$ as the reflection of $\tau_0^-$ over the $y$-axis. Finally we scale the curves $\tau_0$, $\tau_0^-$ and $\tau_0^+$ down with respect to the origin by the factor of $2^{-j}$ to define curves $\tau_j$, $\tau_j^-$ and $\tau_j^+$ respectively. Then the region $U_j$ is defined as the domain bounded by $\tau_j$, $\tau_j^-$, $\tau_j^+$ and $\tau_{j+1}$. By construction, each of the regions $U_j$ is bilipschitz-equivalent to a square of side length $2^{-j}$ with a bilipschitz constant independent of $j$.

We now define the counterpart $V_j$ of $U_j$ in $\yy$. Let $A_j^+ := \varphi(I_j^+)$ be the image arc of $I_j^+$ on $\Gamma_1$ and respectively $A_j^-$ the image arc of $I_j^-$ on $\Gamma_2$. We now connect the four endpoints of the arcs $A_j^\pm$ by curves inside of $\yy$ as follows.
\begin{figure}
\includegraphics[scale=0.5]{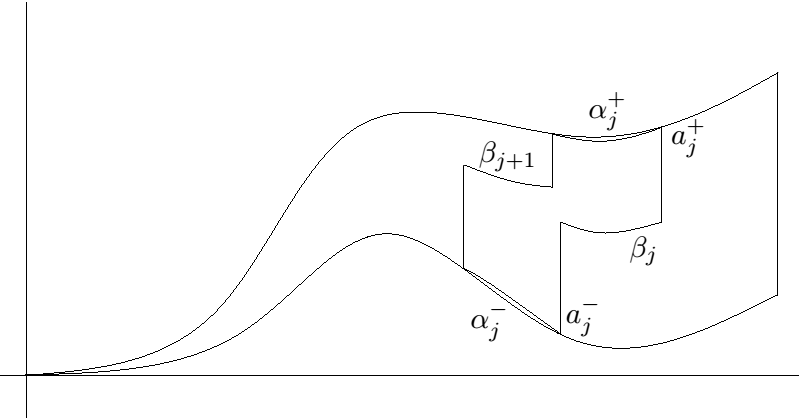}
\caption{Constructing curves between the two graphs.}\label{gammas}
\end{figure}
First, we recall the assumption that $\Gamma_1$ is the graph of a smooth function in some coordinate system. We may suppose that it is the graph of a function $\Phi$ over the interval $[0,1]$ with the point $P$ being at the origin. Now we may in fact assume that the derivative of $\Phi$ at $x = 0$ is zero.  Indeed,  because if the derivative was a different number, say $\Phi'(0) = k$, then at least in a small neighborhood of the origin $\Gamma_1$ could also be written as the graph of a function over the line $y = kx$ instead of the $x$-axis, due to the fact that close enough to $x = 0$ each pair of points on the graph of $\Phi$ must meet at a slope close to $k$. Note that smoothness of the graph is preserved in this change of coordinate system due to the implicit function theorem. Hence by possibly restricting ourselves to a smaller piece of $\Gamma_1$ we may assume that $\Phi'(0)=0$. Due to the same reasoning, we may also suppose that $\Gamma_2$ is a graph of a smooth function $\Psi$ with $\Psi'(0)=0$ over the $x$-axis.

Hence our domain $\yy$ may be assumed to be the one bounded by the graphs of two smooth functions $\Phi$ and $\Psi$ over the interval $[0,1]$, so that $\Phi(0) = \Psi(0) = 0$ and $\Phi(x) > \Psi(x)$ for $x > 0$. Possibly after a bilipschitz change of variables, the curve $\Gamma_3$ is assumed to be the vertical segment from $(1,\Psi(1))$ to $(1,\Phi(1))$. Let us also define  auxiliary curves inside $\yy$ by letting $\gamma_j$ denote the graph of the function $\Phi/j + (1-1/j)\Psi$ over $[0,1]$.

We now define curves $\beta_j$ inside $\yy$ as follows. If $a_j^-$ and $a_j^+$ denote the right endpoints of the arcs $A_j^\pm$, then $\beta_j$ is a crosscut that connects the point $a_j^-$ with $a_j^+$. The exact way how we define $\beta_j$ is by traversing a vertical segment from $a_j^-$ upwards until we hit the curve $\gamma_j$, then traveling along $\gamma_j$ until we are at the $x$-coordinate of the point $a_j^+$, and again going up via a vertical segment to $a_j^+$. Since the curves $\gamma_j$ are pairwise nonintersecting we find that also the curves $\beta_j$ do not intersect each other. Due to the smoothness of $\Psi$ and $\Phi$, the length of $\beta_j$ is comparable to the internal distance between $a_j^+$ and $a_j^-$, which is also comparable to the euclidean distance in this case.

Furthermore, we connect each $a_j^-$ to the other endpoint $a_{j+1}^-$ of $A_j^-$ via a curve $\alpha_j^-$ inside $\yy$ that is close enough to $A_j^-$ so that its length is comparable to the length of $A_j^-$ and it does not intersect any of the $\beta_j$. Again by smoothness, the length of $\alpha_j^-$ is comparable to the internal distance between $a_j^-$ and $a_{j+1}^-$. Similarly we define $\alpha_j^+$ with $+$ in place of $-$ in the previous construction. This way we have defined crosscuts $\beta_j$ and $\alpha_j^\pm$ inside $\yy$ which do not intersect each other and have lengths comparable to the internal distances of their endpoints.

The set $V_j$ is now defined as the region bounded by $\alpha_j^+$, $\alpha_j^-$, $\beta_j$ and $\beta_{j+1}$.
\\\\
In order to define a homeomorphic extension $h : T \to \yy$ of $\varphi$, we first define the map $h$ from $U_j$ to $V_j$. We map each of the four curves which make up the boundary of $U_j$ to the correspording part of $\partial V_j$ by a constant speed map. This will give a Lipschitz boundary homeomorphism from $\partial U_j \to \partial V_j$, which may be extended as a Lipschitz homeomorphism between the interiors with an increase in the Lipschitz norm by at most a constant factor by Theorem~\ref{thm:lip}. The Lipschitz norm of the boundary map is controlled by $2^{-j} |\partial V_j|$. This gives the inequality
\[\int_{U_j} |Dh|^2 \, dx\, dy \leq C|\partial V_j|^2,\] which leads us to find estimates on the perimeter of $V_j$ in terms of the Douglas condition on $\varphi$.

First, we estimate the length of the curve $\alpha_j^-$. We integrate over the sets $I_{j-1}^-$ and $I_{j+1}^-$ to find that
\[\int_{I_{j+1}^-}\int_{I_{j-1}^-} \frac{|\varphi(x) - \varphi(y)|^2}{|x-y|^2} \, dx \, dy \geq c\int_{I_{j+1}^-}\int_{I_{j-1}^-} \frac{|\alpha_j^-|^2}{2^{-2j}}  \, dx \, dy = c|\alpha_j^-|^2.\]
Here we have used the fact that $\Gamma_2$ is a smooth graph to obtain the estimate $|\varphi(x) - \varphi(y)| \geq  c|\alpha_j^-|$ for $x\in I_{j-1}^-$ and $y\in I_{j+1}^-$, as the distance between two points on the graph is always comparable to the difference in their $x$-coordinates. One may replace $-$ by $+$ above to find the same estimate for $|\alpha_j^+|$.

For the curve $\beta_j$, we let $d_j$ denote the minimal distance between the sets $A_j^-$ and $A_j^+$. Then
\[\int_{I_{j}^-}\int_{I_{j}^+} \frac{|\varphi(x) - \varphi(y)|^2}{|x-y|^2} \, dx \, dy \geq c\int_{I_{j}^-}\int_{I_{j}^+} \frac{d_j^2}{2^{-2j}}  \, dx \, dy = c\,d_j^2.\]
However, by triangle inequality we also have the estimate $|\beta_j| \leq d_j + |\alpha_j^-| + |\alpha_j^+|$. All together, this lets us estimate the quantity $|\partial V_j|^2$ from above by a uniform constant times the integral of $|\varphi(x)-\varphi(y)|^2/|x-y|^2$ over the sets $I_{j-1}^- \times I_{j+1}^-$, $I_{j-1}^+ \times I_{j+1}^+$ and $I_{j}^- \times I_{j}^+$. These sets are disjoint in $j$, so by summing up we obtain that
\[\int_{\bigcup_j U_j} |Dh|^2 \, dx \, dy \leq C\sum_j |\partial V_j|^2 \leq C\int_{\partial T}\int_{\partial T} \frac{|\varphi(x) - \varphi(y)|^2}{|x-y|^2} \, dx \, dy.\]
It remains to define the extension $h$ outside the $U_j$ and to control the $\mathscr W^{1,2}$-energy there. For this is enough to consider how $h$ maps the triangle $T_j^- \subset T$ bounded by $I_j^-$ and $\tau_j^-$ to the region in $\yy$ bounded by $\varphi(I_j^-)$ and $\alpha_j^-$. But this part is in fact easy, since we may use the same construction as in the proof of Theorem \ref{thm:generalext}. Let $T_0$ denote the triangle used as the domain of definition in Step 2 of the proof of Theorem \ref{thm:generalext}. Then we simply transform this triangle into $T_j^-$ so that the sides on the real line are mapped to each other, and repeat the same construction done in Step 2 and 3 of Theorem \ref{thm:generalext}.

The $\mathscr W^{1,2}$-energy of $h$ over $T_j^-$ is controlled by the estimate
\[\int_{T_j^-} |Dh|^2 \, dx\, dy \leq C\sum_{D \in \mathcal{D}_j^-} |\varphi(D)|^2,\]
where $\mathcal{D}_j^-$ denotes the collection of all dyadic intervals over the set $I_j^-$. Summing up, we get that
\[\int_{\bigcup_j T_j^-} |Dh|^2 \, dx\, dy \leq C\sum_{D \in \mathcal{D}} |\varphi(D)|^2,\]
where $\mathcal{D}$ denotes the union of all dyadic intervals over the interval $[-1,0]$ which do not contain $0$ as an endpoint (since none of the intervals in $\mathcal{D}_j^-$ do). Using the fact that $[-1,0]$ is mapped to the graph of a smooth function under $\varphi$, for each such dyadic interval $D$ we have the estimate
\[ |\varphi(D)|^2 \leq C\int_{D_{-1}} \int_{D_1}  \frac{|\varphi(x) - \varphi(y)|^2}{|x-y|^2} \, dx \, dy,\]
where $D_{-1}$ and $D_1$ denote the two dyadic neighbors of $D$. Since the sets $D_{-1} \times D_1$ are disjoint when $D$ ranges over all intervals in $\mathcal{D}$, this estimate is summable over $D \in \mathcal{D}$ and yields
\[\int_{\bigcup_j T_j^-} |Dh|^2 \, dx\, dy \leq C \int_{[-1,0]} \int_{[-1,0]}  \frac{|\varphi(x) - \varphi(y)|^2}{|x-y|^2} \, dx \, dy < \infty.\]
The case with $+$'s instead of $-$ is done exactly in the same way. This proves our claim.

\end{proof}


\begin{thebibliography}{0}

\bibitem{Ahre}
Ahlfors,  L. V.  \textit{Quasiconformal reflections},  Acta Math. 109 (1963) 291--301.


\bibitem{AS}
Alessandrini, G. and Sigalotti, M. {\em Geometric properties of solutions to the anisotropic $p$-Laplace equation in dimension two},  Ann. Acad. Sci. Fenn. Math. 26, (2001)  249--266.

\bibitem{Anb}
 Antman, S. S.  {\em Nonlinear problems of elasticity. Applied Mathematical Sciences}, 107. Springer-Verlag, New York, 1995.




\bibitem{AIMb}
 Astala, K.,   Iwaniec, T. and Martin,  G. {\em Elliptic partial differential equations and quasiconformal mappings in the plane}, Princeton University Press, 2009.




\bibitem{Bac}
Ball, J. M.  {\em Convexity conditions and existence theorems in nonlinear elasticity},  Arch. Rational Mech. Anal. {\bf 63} (1976/77), no. 4, 337--403.



\bibitem{Ch}
Choquet, G. {\em Sur un type de transformation analytique
g\'{e}n\'{e}ralisant la repr\'{e}sentation conforme et d\'{e}finie
au moyen de fonctions harmoniques,} Bull. Sci. Math., {\bf 69},
(1945), 156-165.

\bibitem{Cib}
 Ciarlet, P. G. {\em Mathematical elasticity Vol. I. Three-dimensional elasticity}, Studies in Mathematics and its Applications, 20. North-Holland Publishing Co., Amsterdam, 1988.


\bibitem{Do}
Douglas, J.  {\em Solution of the problem of Plateau},  Trans. Amer. Math. Soc. {\bf 33} (1931) 231--321.




\bibitem{Ga}
Gagliardo, E. {\em Caratterizzazioni delle tracce sulla frontiera relative ad alcune classi di funzioni in n variabili},   Rend. Sem. Mat. Univ. Padova 27 (1957), 284--305.

\bibitem{GeMo}
 Gehring, F. W. \textit{Characteristic properties of quasidisks},
S\'eminaire de Math\'ematiques Sup\'erieures [Seminar on Higher Mathematics], 84. Presses de l’Universit\'e de Montr\'eal, Montreal, Que., (1982).



\bibitem{Ge2}
Gehring, F. W. Uniform domains and the ubiquitous quasidisk. - Jahresber. Deutsch. Math.-Verein. {\bf 89}, (1987), 88--103.



\bibitem{GeVa}
Gehring, F. W. and V\"ais\"al\"a, J.
{\em Hausdorff dimension and quasiconformal mappings},  
J. London Math. Soc. (2) {\bf 6} (1973), 504--512.



\bibitem{HP}
 Hencl, S. and  Pratelli, A.  {\em Diffeomorphic approximation of $W^{1,1}$ planar Sobolev homeomorphisms},  J. Eur. Math. Soc. (JEMS) {\bf 20} (2018), no. 3, 597--656. 





\bibitem{IMb}
Iwaniec, T. and  Martin,  G. {\em Geometric Function Theory and Non-linear Analysis}, Oxford Mathematical Monographs, Oxford University Press, 2001.



\bibitem{IMS}
 Iwaniec, T.,  Martin, G. and  Sbordone, C.  $L^p$-integrability \& weak type $L^2$-estimates for the gradient of harmonic mappings of $\mathbb D$. Discrete Contin. Dyn. Syst. Ser. B {\bf 11} (2009), no. 1, 145--152.




\bibitem{IOan}
Iwaniec, T. and Onninen J. {\em  n-harmonic mappings between annuli: the art of integrating free Lagrangians},  Mem. Amer. Math. Soc. 218 (2012).


\bibitem{IOinver}
Iwaniec, T. and Onninen J. {\em   Invertibility versus Lagrange equation for traction free energy-minimal deformations},  Calc. Var. Partial Differential Equations 52 (2015), no. 3-4, 489--496. 

\bibitem{IOmoapprox}
Iwaniec, T. and Onninen J. {\em Monotone Sobolev mappings of planar domains and surfaces}, Arch. Ration. Mech. Anal. 219 (2016), no. 1, 159--181.


\bibitem{IOmo}
Iwaniec, T. and Onninen J. {\em Monotone Hopf-Harmonics}, Arch. Ration. Mech. Anal. 237 (2020), no. 2, 743--777.


\bibitem{Ki}
 Kirszbraun, M. D.  {\em \"Uber die zusammenziehende und Lipschitzsche Transformationen}, Fund. Math. 22, (1934) 77--108.

\bibitem{Kn}
Kneser, H., {\em L\"{o}sung der Aufgabe 41,} Jahresber. Deutsch.
Math.-Verein., {\bf 35}, (1926), 123--124.


\bibitem{KKO}
Koskela, P.,  Koski, A.  and  Onninen, J.  {\em Sobolev homeomorphic extensions onto John domains},  J. Funct. Anal. to appear, arXiv:2004.09669.

\bibitem{KOext}
 Koski, A.  and   Onninen, J. \textit{Sobolev homeomorphic extensions}, J. Eur. Math. Soc. to appear, arXiv:1812.02811.



\bibitem{Kovalev}
Kovalev, L. V. {\em Optimal extension of Lipschitz embeddings in the plane}, Bull. Lond. Math. Soc. 51 (2019), no. 4, 622--632. 


\bibitem{LVb}
Lehto, O. and  Virtanen, K. I. {\em Quasiconformal mappings in the plane},  Springer-Verlag, New York-Heidelberg, (1973).



\bibitem{Le}
Lewy, H.  {\em  On the non-vanishing of the Jacobian in certain
one-to-one mappings,}  Bulletin Amer. Math. Soc., {\bf 42}, (1936),
689--692.

\bibitem{Mor}
 Morrey, C. B.  {\em The Topology of (Path) Surfaces}, Amer. J. Math. {\bf 57} (1935), no. 1, 17--50.

\bibitem{Morrey}
  Morrey, C. B. \textit{Quasi-convexity and the lower semicontinuity of multiple integrals},  Pacific J. Math. {\bf 2}, (1952). 25--53.

\bibitem{NV}
N\"akki, R. and V\"ais\"al\"a: John disks. - Expo. Math. 9, (1991), 3--43.

\bibitem{Pf}
Pfluger, A. {\em  Ueber die Konstruktion Riemannscher Fl\"achen durch Verheftung},  
J. Indian Math. Soc. (N.S.) {\bf 24} (1960), 401--412


\bibitem{Ra}
Rad\'{o}, T.  {\em Aufgabe 41.}, Jahresber. Deutsch. Math.-Verein.,
{\bf 35}, (1926), 49.

\bibitem{Reb}
Reshetnyak, Yu. G.  {\em Space mappings with bounded distortion}, American Mathematical Society, Providence, RI, 1989.

\bibitem{Ro}
Rohde, S. \textit{Quasicircles modulo bilipschitz maps},  Rev. Mat. Iberoamericana 17 (2001), no. 3, 643--659.

\bibitem{Stb}
 Stein, E. M. {\em Singular integrals and differentiability properties of functions}, Princeton Mathematical Series, No. 30, Princeton University Press, Princeton, N.J., (1970).

\bibitem{Ti}
Tienari, M. {\em Fortsetzung einer quasikonformen Abbildung \"uber einen Jordanbogen},  
Ann. Acad. Sci. Fenn. Ser. A I No. 321 (1962). 





\bibitem{Va1.5}
V\"ais\"al\"a, J.  {\em Uniform domains}. - Tohoku Math. J. (2) 40, (1988), 101--118.


\bibitem{Ur}
Ural'ceva, N. N. {\em Degenerate quasilinear elliptic systems}, Zap. Nau\v cn. Sem. Leningrad.
Otdel. Mat. Inst. Steklov. (LOMI) 7, (1968) 184--222.

\bibitem{Ve}
Verchota, G. C. {\em Harmonic homeomorphisms of the closed disc to itself need be in $W^{1,p}, p<2$, but not $W^{1,2}$}, 
Proc. Amer. Math. Soc. {\bf 135} (2007), no. 3, 891--894.

\bibitem{Zh}
 Zhang, Y.R.-Y. {\em Schoenflies solutions with conformal boundary values may fail to be Sobolev},  Ann. Acad. Sci. Fenn. Math.  44 (2019), 791--796.


\end{thebibliography}
\end{document}